\documentclass[a4paper,11pt]{amsart}

\usepackage{multicol}
\usepackage{amsmath,latexsym,amsbsy,amssymb}
\usepackage{enumerate}
\usepackage{amsthm}
\usepackage[latin1]{inputenc}

\newtheorem{Theorem}{Theorem}[section]
\newtheorem{Definition}[Theorem]{Definition} 
\newtheorem{Proposition}[Theorem]{Proposition}
\newtheorem{Corollary}[Theorem]{Corollary}
\newtheorem{Rem}[Theorem]{Remark}
\newtheorem{Lemma}[Theorem]{Lemma}

\numberwithin{equation}{section}

\newcommand{\norm}[1]{\left\|#1\right\|}
\newcommand{\pd}[1]{\langle #1 \rangle}

\newcommand{\R}{{\mathbb R}}

\begin{document}

\title{Non-Lipschitz points and the $SBV$ regularity of the minimum time function}

\author[Giovanni Colombo]{Giovanni Colombo}
\address[Giovanni Colombo]{Universit\`a di Padova, Dipartimento di Matematica, via Trieste 63, 35121 Padova, Italy}
\email{colombo@math.unipd.it}

\author[Khai T. Nguyen]{Khai T. Nguyen}
\address[Khai T. Nguyen]{Universit\`a di Padova, Dipartimento di Matematica, via Trieste 63, 35121 Padova, Italy}
\email{khai@math.unipd.it}

\author[Luong V. Nguyen]{Luong V. Nguyen}
\address[Luong V. Nguyen]{Universit\`a di Padova, Dipartimento di Matematica, via Trieste 63, 35121 Padova, Italy}
\email{vanluong@math.unipd.it}
\thanks{This work was partially supported by M.I.U.R., project
``Viscosity, metric, and control theoretic methods for nonlinear
partial differential equations'', by CARIPARO Project ``Nonlinear Partial Differential Equations:
models, analysis, and control-theoretic problems'',  and University of Padova research project ``Some analytic and
differential geometric aspects in Nonlinear Control Theory, with applications to Mechanics''. The authors are also supported by
the European Union under the 7th Framework Programme ``FP7-PEOPLE-2010-IT'', Grant agreement number 264735-SADCO.
In particular, the third author is a SADCO PhD fellow, position ESR4. K.T.N is also supported by the ERC Starting
Grant 2009 n.240385 ConLaws.}



\date{\today}
\begin{abstract}
This paper is devoted to the study of the Hausdorff dimension of the singular set of the minimum time function $T$
under controllability conditions which do not imply the Lipschitz continuity of $T$.
We consider first the case of normal linear control systems with constant
coefficients in $\mathbb{R}^N$.
We characterize points around which $T$ is not Lipschitz as those which can be reached from the origin by an
optimal trajectory (of the reversed  dynamics) with vanishing minimized Hamiltonian. Linearity permits an explicit representation
of such set, that we call $\mathcal{S}$. Furthermore, we show that $\mathcal{S}$ is $\mathcal{H}^{N-1}$-rectifiable with positive
$\mathcal{H}^{N-1}$-measure. Second, we 
consider a class of control-affine \textit{planar} nonlinear systems satisfying a second order
controllability condition: we characterize the set $\mathcal{S}$
in a neighborhood of the origin in a similar way and prove the $\mathcal{H}^1$-rectifiability of $\mathcal{S}$ and that
$\mathcal{H}^1(\mathcal{S})>0$. In both cases, $T$ is known to 
have epigraph with positive reach, hence to be a locally $BV$ function (see \cite{CMW,GK}).
Since the Cantor part of $DT$ must be concentrated in $\mathcal{S}$, our analysis yields that $T$ is $SBV$, i.e.,
the Cantor part of $DT$ vanishes. Our results imply also that $T$ is locally of class $\mathcal{C}^{1,1}$ outside a
$\mathcal{H}^{N-1}$-rectifiable set. With small changes, our results are valid also in the case of multiple control input.
\end{abstract}
\maketitle
\section{Introduction}\label{intro}
Consider the control system
\begin{equation}\label{introdyn}
\begin{cases}
\dot{x} &=\;\, F(x) + G(x)u,\qquad |u|\le 1,\; x\in\mathbb{R}^N,\\
x(0)&=\;\,\xi.
\end{cases}
\end{equation}
The minimum time function $T(\xi)$ to reach the origin from $\xi$ under the above dynamics is well known to be,
in general, both non-smooth and non-Lipschitz. Several papers were devoted to the partial regularity of $T$. In particular,
we quote results devoted to establishing (semi)convexity/concavity properties of $T$ under various assumptions (see
\cite{CS0,CS,CMW,GK,K,CaKh,camapw}), together with \cite{BP}, which is concerned with planar systems.

In this paper we concentrate mainly on the lack of Lipschitz continuity of $T$, which is essentially due to the lack of first
order controllability. More precisely, even if at some $x$ the right hand side of \eqref{introdyn} does not point towards the origin, i.e., the
scalar product between any vector in $F(x)+G(x)\mathcal{U}$ and $x$ is merely nonnegative, it is still possible that a trajectory
through $x$ reach the origin, provided the Lie bracket $[F(x),G(x)]$ has nonvanishing scalar product with the missing direction
$x$ (higher order controllability).
The price to pay is a slower approaching to the origin: one needs to switch between $G$ and $-G$, like a sailor which has to beat to windward.
The simple example $\ddot{x}=u\in [-1,1]$ exhibits this behavior: at every point of the $x_1$-axis (we set $\dot{x}_1=x_2$,
$\dot{x}_2=u$) the right hand side of \eqref{introdyn} is vertical and $T$ is not locally Lipschitz in the whole of $\mathbb{R}^n$
(but at the points of the $x_1$-axis \textit{is} Lipschitz). By introducing the minimized Hamiltonian
\[
h(x,\zeta):=\langle F(x),\zeta\rangle + \min_{u\in\mathcal{U}}\langle G(x)u,\zeta\rangle,
\]
the condition of non-pointing towards the origin becomes
\[
h(x,\zeta)\ge 0,
\]
where $\zeta$ is a normal to the sublevel of $T$ corresponding to $T(x)$. Since
the minimized Hamiltonian is constant and nonpositive along every optimal trajectory,
%for the backwards adjoint curve starting \textit{at any such normal vector}, 
it is natural to expect that non-Lipschitz points of $T$ lie exactly where such Hamiltonian vanishes.
In fact, in Section \ref{sec:minham} we prove this characterization.

Let $\mathcal{S}$ be the set of non-Lipschitz points of $T$. In Sections \ref{sec:lin} and \ref{sec:nonlin} we characterize $\mathcal{S}$ using
points which belong to an optimal pair (i.e., an optimal trajectory together with a corresponding adjoint arc)
of \eqref{introdyn} with vanishing Hamiltonian, and, for the
linear case, we give an explicit representation of $\mathcal{S}$. As a consequence, we show that at each $\bar{x}\in\mathcal{S}$
the sublevel $\mathcal{R}_{T(\bar{x})}$ is tangent to $\mathcal{S}$, in the sense that there exists a normal vector to $\mathcal{R}_{T(\bar{x})}$
at $\bar{x}$ which is tangent to the optimal trajectory reaching $\bar{x}$ from the origin.
The result is valid for both normal linear systems with constant
coefficients in any space dimension (see Theorem \ref{charnonlip}) and for smooth nonlinear two dimensional systems
such that the origin is an equilibrium point, the linearization at $0$ is normal and furthermore $DG(0)=0$ (see Theorem \ref{rect2dim}).
The condition $DG(0)=0$ ensures that the nonlinearity is sufficiently mild to preserve a linear like behavior in a neighborhood of the
origin whose size can be estimated. In both cases it is known that the epigraph of $T$ has locally positive reach (see \cite{CMW,GK}).
Reasons for the restriction to two space dimensions
in the nonlinear case are discussed in the paper \cite{GK}, to which the present work owes some results.

Our first result is the $\mathcal{H}^{N-1}$-rectifiability of $\mathcal{S}$ for the linear single
input case, see Theorem \ref{N-1rect}, and, respectively, the $\mathcal{H}^1$-rectifiability for the nonlinear two dimensional case, see Theorem \ref{rect2dim}.
For the linear case, the switching function
\[
g_\zeta (t) = \langle \zeta , e^{At}b\rangle,
\]
where $b$ is a column of the matrix $B$, plays an important role. Actually we partition $\mathcal{S}$ according to the
multiplicity of zeros of $g_\zeta$ and embed each part into a locally Lipschitz graph. The nonlinear case 
%requires first
%a slight improvement of Pontryagin's Maximum Principle, showing its validity \textit{for every} adjoint vector whose
%endpoint is normal to the sublevel of $T$ corresponding to $T(x)$ (see the Claim in the proof of Theorem \ref{rect2dim}). This result
%permits to 
is handled by showing that $\mathcal{S}$ consists of optimal trajectories with vanishing Hamiltonian. Since, due to the space dimension
restriction, such trajectories are at most two for the single input case, the $\mathcal{H}^1$-rectifiability is clear. We observe that
the investigation of the regularity of the \textit{Minimum time front}, i.e., the boundary of $\mathcal{R}_t$, performed in \cite[Chapter 3]{BP}
cannot provide information on non-Lipschitz points of $T$, since an analysis of sublevels is not enough to describe the epigraph
of a function.

These rectifiability results shade some light on the propagation of singularities for minimum time functions.
In fact, the positive reach property of epi$(T)$ implies that $T$ is locally semiconvex outside the closed set
$\mathcal{S}$ (see \cite[Theorem 5.1]{CM}).
The structure of singularities of semiconvex functions is well understood (see \cite[Chapter 4]{CS}): in particular, a locally
semiconvex function is of class $\mathcal{C}^{1,1}$ outside a $\mathcal{H}^{N-1}$-rectifiable set.
Therefore our rectifiability results for $\mathcal{S}$
imply that $T$ is of class $\mathcal{C}^{1,1}$ outside a closed $\mathcal{H}^{N-1}$-rectifiable set. We observe that, for general functions
whose epigraph satisfies a uniform external sphere condition, the Hausdorff dimension of the set of non-Lipschitz points was proved to
be less or equal to $n-1/2$, with an example showing the sharpness of the estimate (see \cite[Theorem 1.3 and Proposition 7.3]{AKD}).
The present paper therefore improves that result, for the particular case of a minimum time function. We prove also (see section \ref{sec:propag})
a converse (propagation) result: for $\mathcal{H}^{N-1}$-a.e.~$x\in \mathcal{S}$ such that $T(x)$ is small enough
there exists a neighborhood $V$ such that $\mathcal{H}^{N-1}(\mathcal{S}\cap V)>0$. In particular, for single input normal linear systems
we show that, for all $t>0$ small enough, the set $\mathcal{S}\cap\mathcal{R}_t$, up to a $\mathcal{H}^{N-2}$-negligible subset, is a $\mathcal{C}^{1}$-surface
of dimension $N-2$. 
Apparently, this is the first result in the literature concerning propagation of non-Lispchitz singularities.

Applications of the above results to time optimal feedbacks will be discussed elsewhere.

The positive reach property of epi$(T)$ implies also that $T$ has locally bounded variation (see \cite[Proposition 7.1]{CM}).
Therefore, as a consequence of the above analysis we obtain that $T$ belongs to the smaller class of locally $SBV$ functions,
namely the Cantor part $D_cT$ of its distributional derivative, which is a Radon measure by definition of $BV$, vanishes. In fact,
on one hand $D_cT$ must be concentrated on the set $\mathcal{S}$ of non-Lipschitz points of $T$, on the other the rectifiability
properties that we proved for $\mathcal{S}$ yield exactly the $SBV$ regularity of $T$.
To our best knowledge this property of $T$ is observed here for the first time.
\section{Preliminaries}\label{sec:prelim}
\subsection{Nonsmooth analysis, sets with positive reach, and geometric measure theory}\quad\\
The space dimension is denoted by $N$ and we suppose $N \ge 2$. The unit sphere in $\mathbb{R}^N$ is denoted by $\mathbb{S}^{N-1}$.\\
Let $K\subset\R^N$ be closed with boundary $\mathrm{bdry} K$.
We need some concepts of nonsmooth analysis (see, e.g. \cite[Chapters 1 and 2]{CLSW}). Given $x \in K$ and $v \in \R^N$,
we say that $v$ is a \emph{proximal normal} to $K$ at $x$, and denote this fact by $v\in N_K(x)$,
provided there exists $\sigma=\sigma(v,x)\ge 0$ such that
\[
\pd{  v,y-x}\le\sigma \norm{y-x}^2,\quad \textrm{ for all }y\in K.
\]
If $K$ is convex, then $N_K(x)$ coincides with the normal cone of Convex Analysis.
The set of \emph{limiting normals} to $K$ at $x$ is denoted by $N^L_{K}(x)$, and consists of those $v\in\mathbb R^N$ for which
there exist sequences $\{x_i\},\,\{v_i\}$ with $x_i\to x$, $v_i\to v$, and $v_i\in N_{K}(x_i)$.
The \emph{Clarke normal} cone $N^C_K(x)$ equals $\overline{\mathrm{co}}N^{L}_{K}(x)$, where \lq\lq$\overline{\mathrm{co}}$\rq\rq\
means the closed convex hull. \\
Let $\Omega \subset \R^N$ be open and let $f: \Omega \rightarrow \R \cup\{+\infty\}$ be lower semicontinuous. 
The epigraph of $f$ is $\mathrm{epi}(f) = \{(x,y) \in \Omega \times \R: y\ge f(x)\}$. The \textit{proximal subdifferential} 
$\partial f(x)$ of $f$ at a point $x$ of dom$(f)=\{ x:f(x)<+\infty\}$ is the set of vectors $v \in \R^N$ such that 
\[
(v,-1)\in N_{\text{epi}(f)}(x,f(x)).
\]
The \textit{horizon subdifferential} $\partial ^{\infty} f(x)$ of $f$ at a point $x\in\mathrm{dom}(f)$ is the set of vectors $v\in \R^N$ such that
\[
(v,0)\in N_{\text{epi}(f)}(x,f(x)).
\]
This concept is connected with the lack of Lipschitz continuity of $f$ around $x$ (see, e.g., \cite[Chapter 9]{RW}) and will be used mainly in Section 3.\\
Sets with positive reach will play an important role in the sequel. The definition was first given by Federer 
in \cite{Fe} and later studied by several authors (see the survey paper \cite{CT}).
\begin{Definition} \label{posreach} 
Let $K\subset \R^N$ be locally closed. We say that $K$ has locally positive reach provided there
exists a continuous function $\varphi:K\to[0,+\infty)$ such that the inequality
\begin{equation}\label{ineqphi}
\langle v,y-x\rangle\le\varphi(x)\| v\|\,\| y-x\|^2
\end{equation}
holds for all $x,y\in K$ and $v\in N_{K}(x)$.
\end{Definition}
In particular, every convex set has positive reach: it suffices to take $\varphi \equiv 0$ in (\ref{ineqphi}).\\
% Among several good properties satisfied by sets with positive reach, we recall here the following one which follows immediately from
% (\ref{ineqphi}).
% \begin{Proposition}
% Let $K \subset \R^N$ have locally positive reach and let the sequences $\{x_n: n \in \N\} \subset K$ 
% and $\{v_n: n \in \N\} \subset \R^N$ be such that
% \[
% \lim_{n\to \infty}{x_n} = x \in K, \quad v_n \in N_K(x_n),\quad \lim_{n\to \infty}{v_n} = v.
% \]
% Then $v \in N_K(x).$
% \end{Proposition}
Continuous functions whose epigraph has locally positive reach will be crucial in our analysis. 
Such functions enjoy several regularity properties, mainly studied in \cite{CM}. We list two of them which will be used in the sequel.
\begin{Theorem} \label{fposreach} Let $\Omega \subset \R^N$ be open and let $f: \Omega \rightarrow \R$ be 
continuous and such that epi$(f)$ has locally positive reach. Then
\begin{enumerate}
\item[(i)] $f$ is a.e. differentiable in $\Omega$,
\item[(ii)] $f$ has locally bounded variation in $\Omega$.
\end{enumerate}
\end{Theorem} 
A brief survey of basic notions on functions with bounded variation will be presented in Section 6.
Our main reference, also for other basic concepts in geometric measure theory, is \cite{AFP}.
% We recall now a basic concept from geometric measure theory. Let $0\le k < \infty$ 
% and let $\mathcal{H}^k$ denote the Hausdorff $k$-dimensional measure (see Definition 2.46 in \cite{AFP}). 
% Let $E$ be $\mathcal{H}^k$-measurable. We say that $E$ is countably $k$-rectifiable if there exist 
% countably many Lipschitz functions $f_i: \R^k \rightarrow \R^N$ such that
% \[
% E \subseteq \bigcup_{i=1}^{\infty}f_i(\R^k).
% \]
% By the extension theorem for Lipschitz functions, the above condition is equivalent to the following one: there exist 
% countably many sets $A_i \subseteq \R^k$ and countably many Lipschitz functions $f_i: A_i \rightarrow \R^N$ such that
% \[
% E \subseteq \bigcup_{i=1}^{\infty}f_i(A_i).
% \]
\subsection{Control theory}
Consider the following autonomous control system
\begin{equation}\label{generalsys}
\left\{\begin{array}{ll}
\dot{y}(t)\: = \: f(y(t),u(t)) & \text{a.e.,}\\
u(t) \: \in \:  \mathcal{U} & \text{a.e.,}\\
y(0)  \: = \:  x,
\end{array}\right.
\end{equation}
where the control set $\mathcal{U} \subset \R^M$ is nonempty and compact and $f: \R^N \times \mathcal{U} \rightarrow \R^N$ is continuous and 
Lipschitz with respect to the state variable $x$, uniformly with respect to $u$.
% satisfies
% \begin{equation}\label{condf}
% \|f(x,u) - f(y,u)\| \le L\|x-y\|,\,\,\forall x,y \in \R^N, \forall u \in \mathcal{U},
% \end{equation}
% for a positive constant $L$.
We denote by $\mathcal{U}_{ad}$ the set of admissible controls, i.e., all measurable functions $u$ such that $u(s) \in \mathcal{U}$ for
a.e.~$s$. Under our assumptions, for any $u(\cdot) \in \mathcal{U}_{ad}$, there is a unique Carath\'eodory 
solution of (\ref{generalsys}) denoted by $y^{x,u}(\cdot)$. The solution $y^{x,u}(\cdot)$ is called the trajectory starting 
from $x$ corresponding to the control $u(\cdot)$.

We will focus mainly on control systems which are linear or nonlinear with respect to the space variable and linear and symmetric
with respect to the control. More precisely, we will consider the linear control system 
\begin{equation}\label{lin}
\left\{\begin{array}{ll}
\dot{y}(t)\: = \: Ay(t)+Bu(t) & \text{a.e.,}\\
u(t) \: \in \:  \mathcal{U} =[-1,1]^{M}& \text{a.e.,}\\
y(0)  \: = \:  x,
\end{array}\right.
\end{equation}
where $1\le M\le N$ and $A\in\mathbb{M}_{N\times N}$, $B\in\mathbb{M}_{N\times M}$ and $\mathcal{U}=[-1,1]^{M}\ni (u_1,\ldots,u_M)=:u$,
together with the nonlinear two dimensional control system
\begin{equation}\label{nonlin}
\left\{\begin{array}{ll}
\dot{y}(t)\: = \: F(y(t))+G(y(t))u(t) & \text{a.e.,}\\
u(t) \: \in \:  [-1,1]^{M} & \text{a.e.,}\\
y(0)  \: = \:  x,
\end{array}\right.
\end{equation}
where $F$ and $G$ are suitable vector fields (the actual assumptions will be stated later) and $1\le M\le 2$.
We will use also the notation $B=(b_1,\ldots ,b_M)$ or $G=(g_1,\ldots,g_M)$, where each entry is an $N$-dimensional column.
%For any $u(\cdot)\in\mathcal{U}_{ad}$, the solution of (\ref{lin}) is
% $
% y^{x,u}(t)=e^{At}x+ \int_0^te^{A(t-s)}Bu(s)\, ds.
% $
Note that $\bar{x}$ is reachable by a solution of (\ref{lin}) at time $t$ if and only if the following (equivalent) conditions hold:
\begin{equation}\label{equivlin}
\bar{x}=e^{At}x+ \int_0^te^{A(t-s)}Bu(s)\,ds\qquad\text{and}\qquad x = e^{-At}\bar{x} - \int_0^te^{-As}Bu(s)\,ds,
\end{equation}
where $u(\cdot)\in\mathcal{U}_{ad}$.\\
For a fixed $x\in\mathbb{R}^N$, we define
$$\theta(x,u):=\min\ \lbrace{t\geq 0\: |\: y^{x,u}(t)=0\rbrace}.$$
Of course, $\theta(x,u)\in [0,+\infty]$, and $\theta(x,u)$ is the time taken for the trajectory
$y^{x,u}(\cdot)$ to reach $0$, provided $\theta(x,u)<+\infty$.
The \textit{minimum time} $T(x)$ to reach $0$ from $x$ is defined by
$
%\begin{equation}\label{minT}
T(x):=\inf\ \lbrace{\theta(x,u)\: |\: u(\cdot)\in\mathcal{U}_{\mathrm{ad}}\rbrace}
%\end{equation}
$
and under standard assumptions the infimum is attained.
% $f(z,\mathcal{U})$ is convex for every $z\in \R^N$ then
% \begin{equation}\label{minT1}
% T(x):=\min\ \lbrace{\theta(x,u)\: |\: u(\cdot)\in\mathcal{U}_{\mathrm{ad}}\rbrace}.
% \end{equation}
A minimizing control, say $\bar{u}(\cdot)$, is called an \textit{optimal control}. 
The trajectory $y^{x,\bar{u}}(\cdot)$ corresponding to $\bar{u}(\cdot)$ is called an \textit{optimal trajectory}.

Denote by $\mathcal{R}_t$ the set of points which can be steered to the origin with the control dynamics (\ref{generalsys}) 
within the time $t$. Then $\mathcal{R}_t$ is the set of points which can be reached from the origin with the \textit{reversed dynamics}
\begin{equation}\label{generalrev}
\left\{\begin{array}{ll}
\dot{x}(t)\: = \:- f(x(t),u(t)) & \text{a.e.,}\\
u(t) \: \in \:  \mathcal{U} & \text{a.e.,}\\
x(0)  \: = \:  0,
\end{array}\right.
\end{equation}
i.e., $\mathcal{R}_t$ is the sublevel $\{x\in \R^N: T(x) \le t\}$ of $T(\cdot)$.
If $\bar{u}$ is an admissible control steering $x$ to the origin in the minimum time $T(x)$, then the Dynamic Programming 
Principle 
%(see, e.g., Proposition 2.1, Chapter IV, in \cite{BCD})
implies that  
%$T(\cdot)$ is strictly increasing along the optimal trajectory $y^{x,\bar{u}}$. Therefore, 
for all $0<t<T(x)$ the point $y^{x,\bar{u}}(t)$ belongs to the boundary of $\mathcal{R}_{t}$.

We state now Pontryagin's Maximum Principle,
%  for points belonging to the boundary of reachable sets. In view of the previous remark,
% this will apply also to
% points belonging to optimal trajectories. We give
giving first its linear version for the special case we are interested in.
\begin{Theorem}\label{th:PMPlin}
Consider the problem \eqref{lin} under the following (normality) condition:\\
for every column $b_i$ of $B$, $i = 1,...,M$, we have
$$\mathrm{rank}[b_i,Ab_i,...,A^{N-1}b_i] = N. $$
Let $T>0$ and let $x \in \R^N$. The following statements are equivalent:
\begin{itemize}
\item[(i)] $x \in \mathrm{bdry} \mathcal{R}_T$,
\item[(ii)] there exists an optimal control $\bar{u}$ steering $x$ to the origin in time $T$; in particular, $T(x) = T$;
\item[(iii)]\text{(Pontryagin Maximum Principle)} for every $\zeta \in N_{\mathcal{R}_T}(x), \zeta \ne 0$, we have
\begin{equation} \label{max}
\bar{u}_i(t) = -\mathrm{sign}\left(\langle\zeta,e^{-At}b_i\rangle\right) \,\,\text{a.e.}\,\,t\in [0,T], 
\end{equation} 
for all $i=1,2,\ldots ,M$.
\end{itemize}
\end{Theorem}
\noindent A well known reference for this result is \cite[Sections 13 - 15]{hls}.
\begin{Rem}
If the system (\ref{lin}) is normal then $(A,B)$ satisfies the Kalman rank condition. Therefore the minimum time function 
is everywhere finite and continuous (actually H\"older continuous with exponent $1/N$, see, e.g., Theorem 17.3 in \cite{hls} 
and Theorem 1.9, Chapter IV, in \cite{BCD} and references therein).
\end{Rem}
\noindent Before stating Pontryagin's Principle for the nonlinear case \eqref{generalsys},
we need to introduce the minimized Hamiltonian.
We define for every triple $(x,p,u)\in\R^N\times\R^N\times [-1,1]^M$, the \textit{Hamiltonian}:
\[
 \mathcal{H}(x,p,u)=\langle p, f(x,u)\rangle
\]
and the \textit{minimized Hamiltonian}:
\[
h(x,p)=\min\{ \mathcal{H}(x,p,u) : u\in [-1,1]^M \}. 
\]
Observe that if $\bar{x}$ is steered to the origin with respect to the system (\ref{generalsys}) by the control $\bar{u}(\cdot)$ 
in the time $T$, then the origin is steered to $\bar{x}$ with respect to the reversed dynamics (\ref{generalrev}) 
in the same time $T$ by the control $\tilde{u}(t) := \bar{u}(T-t)$. The corresponding trajectory will be denoted by $\bar{y}(t) := y^{\bar{x},\bar{u}} (T-t)$.
Then Pontryagin's Principle reads as follows.
\begin{Theorem}[Pontryagin's Maximum Principle for nonlinear systems]\label{th:PMPnonlin}
Fix $T > 0$ and let $\bar{x} \in \R^N$ together with an optimal control  steering $\bar{x}$ to the origin in the time $T$. 
Then there exists an absolutely continuous function $\lambda: [0,T] \rightarrow \R^N$, never vanishing, such that
\begin{itemize}
\item[(i)] $\dot{\lambda}(t) = \lambda(t)D_xf(\bar{y}(t),\bar{u}(T-t))$, a.e. $t\in [0,T]$,
\item[(ii)] $\mathcal{H}\left(\bar{y}(t),\lambda(t),\bar{u}(T-t)\right) = h\left(\bar{y}(t),\lambda(t)\right)$, a.e. $t\in [0,T]$,
\item[(iii)] $h\left(\bar{y}(t),\lambda(t)\right) = \text{constant}$, for all $t\in [0,T]$,
\item[(iv)] $\lambda(T) \in N^C_{\mathcal{R}_T}(\bar{x})$.
\end{itemize}
\end{Theorem}
\noindent This formulation of the Maximum principle can be obtained by using the classical one (see, e.g., \cite[Theorem 8.7.1]{vin}) 
for the reversed dynamics.
Observe that the condition (\ref{max}) is equivalent to the minimization condition (ii) for the case of linear systems.

\section{Properties connected with the minimized Hamiltonian}\label{sect:ham}
This section is mainly technical and consists of two subsections.
\subsection{Minimized Hamiltonian and normals to epi$(T)$} \label{sect:normals}
This section is concerned with a relation between normals 
to the sublevels of $T$ and normals to epi$(T)$, which was one of the main tools used in \cite{CMW,GK} in order to prove that 
epi$(T)$ has positive reach. We give here a unified and slightly generalized presentation, in order to use it in the sequel.

We recall first that a point $x \in \R^N \setminus\{0\}$ is defined to be an \emph{optimal point} for (\ref{nonlin}) if there exist 
$x_1$ such that $T(x_1) > T(x)$ and a control $u$ with the property that $y^{x_1,u}(\cdot)$ steers $x_1$ to the origin in the 
optimal time $T(x_1)$ and $y^{x_1,u}\left(T(x_1)-T(x)\right) = x$, i.e., if there exists an optimal trajectory for (\ref{nonlin}) which passes through $x$.
It is easy to see, using Pontryagin's Maximum Principle, that for a normal linear system every point is optimal.
Indeed, it is enough to extend the adjoint vector and choose a control which maximizes the Hamiltonian (for the reversed dynamics). 
For nonlinear two dimensional systems, sufficient conditions for this property were given in \cite[Theorem 6.2]{GK}.\\
Second, we recall that sublevels of $T$ are convex in the case of linear systems (see, e.g, \cite[Lemma 12.1]{hls}) and so,
in particular, they have positive reach. For nonlinear two dimensional systems, sufficient conditions for this property were given
in \cite[Theorem 5.1(c)]{GK}.\\
After the above preliminaries, we can state our result. The assumptions are indeed strong, but we emphasize the fact that they
are all satisfied in the two cases we are going to consider in the present paper (see Theorem 3.7
in \cite{CMW} and Theorem 6.5 in \cite{GK}).
\begin{Proposition} \label{generalhamilt}
Consider the general control system (\ref{generalsys}) with the following assumptions:
\begin{itemize}
\item[(i)] $\mathcal{U} \subset \R^M$ is compact and $\{f(x,u): u \in \mathcal{U}\}$ is convex for every $x \in \R^N$.
\item[(ii)] $f: \R^N \times \mathcal{U} \rightarrow \R^N$ is continuous and satisfies
$$ \|f(x,u) -f(y,u)\| \le L\|x-y\|, \,\,\, \forall x,y \in \R^N, u\in \mathcal{U},
$$ for a positive constant $L$. Moreover, the differential of $f$ with respect to the $x$ variable, $D_xf$, exists everywhere, 
is continuous with respect to both $x$ and $u$ and satisfies 
$$ \|D_xf(x,u) - D_xf(y,u)\| \le L_1\|x-y\|,\,\,\, \forall x,y \in \R^N, u\in \mathcal{U},
$$for a positive constant $L_1$.
\end{itemize}
Let $x\in\mathbb{R}^N\setminus\{ 0\}$ and let $T(x)$ be the minimum time to reach the origin from $x$.
Assume that there exists a neighborhood $\mathcal{V}$ of $x$
such that
\begin{itemize}
\item[(1)] $T$ is finite and continuous in $\mathcal{V}$,
\item[(2)] every $y\in \mathcal{V}$ is an optimal point,
\item[(3)] for every $y\in \mathcal{V}$ the optimal control steering $y$ to the origin is unique and bang-bang with finitely many switchings,
\item[(4)] there exists $r>T(x)$ such that $\mathcal{R}_t$ has positive reach for all $t<r$.
\end{itemize}
Let 
%$r>T(x)$ together with $(\tilde{x}(\cdot),\tilde{u}(\cdot))$ be an optimal pair defined on $[0,r]$ steering 0 to $x$ 
%in the optimal time $T(x)$ for the reversed dynamics (\ref{generalrev}) and let
$\zeta \in N_{\mathcal{R}_{T(x)}}(x)$. 
%together with an absolutely continuous function $\lambda:[0,r] \rightarrow \R^N$  for which
% \begin{equation*}
% \left\{\begin{array}{ll}
% \dot{\lambda}(t)\: = \: \lambda(t) D_xf(\tilde{x}(t),\tilde{u}(t)),\\
% \lambda(T(x))  \: = \:  \zeta.
% \end{array}\right.
% \end{equation*}
% and $h(\tilde{x}(t),\lambda(t)) = \langle \lambda(t), f(\tilde{x}(t),\tilde{u}(t)) \rangle= h(x,\zeta)$ for all $t \in [0,r]$. 
Then
\begin{itemize}
\item[(a)] $h(x,\zeta) \le 0$.
\item[(b)] $\left(\zeta,h(x,\zeta)\right) \in N_{\mathrm{epi}(T)}(x,T(x))$.
\end{itemize}
\end{Proposition}
\begin{proof}
Since $\zeta \in N_{\mathcal{R}_{T(x)}}(x)$ and $\mathcal{R}_{T(x)}$ has positive reach, there exists a constant $\sigma \ge 0$ such that
\[
\langle\zeta,y-x\rangle \le \sigma \|\zeta\|\,\|y-x\|^2,
\]
for all $y \in \mathcal{R}_{T(x)}$.\\
Let $(\bar{x}(\cdot),\bar{u}(\cdot))$ be an optimal pair for $x$ %, defined up to some $r> T(x)$
and let $\lambda:[0,T(x)] \rightarrow \R^N$ be absolutely continuous and such that
\begin{equation}\label{adj}
\left\{\begin{array}{ll}
\dot{\lambda}(t)\: = \:\lambda(t)D_xf(\bar{x}(T(x)-t),\bar{u}(T(x)-t)) ,\\
\lambda(T(x))  \: = \:  \zeta.
\end{array}\right.
\end{equation}
and $h(\bar{x}(t),\lambda(t)) = \langle \lambda(t), f(\bar{x}(T(x)-t),\bar{u}(T(x)-t)) \rangle= h(x,\zeta)$ for all $t \in [0,T(x)]$
(see Theorem 3.1 in \cite{CFS}).

Observing that $\bar{x}(t)\in\mathcal{R}_T(x)$ for all $0 <t < T(x)$,
we have
\[
\langle\zeta,\bar{x}(t)-x\rangle \le \sigma \|\zeta\|\,\|\bar{x}(t) -x\|^2.
\]
By using Gronwall's lemma, there is a suitable $M_1>0$ such that $\|\bar{x}(t) - x\| \le M_1t$.
Since $f$ is Lipschitz with respect to $x$, we have for some $M_2>0$ and for all $0<t<T(x)$,
% \[
% \Big\langle\zeta,\int_{0}^{t}{f(\bar{x}(s),\bar{u}(s)ds}\Big\rangle \le M_2\|\zeta\|t^2.
% \]
% %Since $f$ is Lipschitz with respect to $x$, the above inequality yields 
% Hence
\[
\Big\langle \zeta, \frac{1}{t}\int_{0}^{t} {f(x,\bar{u}(s)ds}\Big\rangle \le M\|\zeta\|t.
\]
Taking the upper limit, we obtain
\[
\limsup_{t\to 0+}\Big\langle \zeta, \frac{1}{t} \int_{0}^{t} {f(x,\bar{u}(s)ds}\Big\rangle \le 0,
\]
which implies
\[
\langle \zeta,f(x,\tilde{u})\rangle \le 0,
\]
for some $\tilde{u} \in \mathcal{U}$, since $\{f(x,u): u\in \mathcal{U}\}$ is convex and $\mathcal{U}$ is compact,
whence $h(x,\zeta) \le 0$.

We are now going to show that $(\zeta,h(x,\zeta)) \in N_{\mathrm{epi}(T)}(x,T(x))$, i.e., there is a $\sigma \ge 0$ such that
\begin{equation}\label{norepi}
\langle (\zeta,\vartheta),(y,\beta) - (x,T(x)\rangle \le \sigma (\|y-x\|^2 + |T(x) - \beta|^2),
\end{equation}
for all $(y,\beta)$ in a neighborhood of $(x, T(x))$, say $(y,\beta)\in\mathcal{V}\times [0,r]$, where $\vartheta := h(x,\zeta)$, $\beta\ge T(y)$,
and $r > T(x)$ is such that $T(y) < r$ for all $y\in \mathcal{V}$.
There are two possible cases:
\begin{itemize}
\item[(i)] $T(y) \le T(x)$,
\item[(ii)] $T(y) > T(x)$.
\end{itemize}
In the first case, since $y\in \mathcal{R}_{T(x)}$ and $\mathcal{R}_{T(x)}$ has positive reach, there is a $K_1 \ge 0 $ such that
\[ 
\langle\zeta, y-x\rangle \le K_1\|\zeta\|\,\|y-x\|^2.
\]
Since $\vartheta =h(x,\zeta) \le 0$, if $\beta \ge T(x)$ then (\ref{norepi}) is satisfied.
If instead $\beta < T(x)$, then we set $x_1 = \bar{x}(T(x) - \beta)$. By Gronwall's lemma, there is some $K > 0$ such that 
\[
\|x-x_1\| = \|x - \bar{x}(T(x) - \beta)\| \le K|T(x) - \beta|.
\]
We have
\begin{eqnarray*}
\langle\zeta, y-x\rangle &=& \langle \lambda(\beta), y - x_1\rangle + \langle \lambda(T(x))-\lambda(\beta), y-x_1\rangle + \langle\zeta, x_1 - x\rangle\\
&=&: (I) + (II) +(III).
\end{eqnarray*}
We now consider $(I)$. Since $y\in \mathcal{R}_{\beta}$ and
$\lambda(\beta) \in N^C_{\mathcal{R}_\beta}(x_1)$ (see Definition 2.3 and Corollary 4.8 in \cite{fra}), 
owing to the fact that $\mathcal{R}_{\beta}$ has positive reach, there exist $K_2, K_3 > 0$ such that
\begin{eqnarray*}
(I) &\le& K_2\|\lambda(\beta)\|\,\|y-x_1\|^2\\
&\le & 2 K_2\|\lambda(\beta)\|(\|y-x\|^2 +\|x-x_1\|^2)\\
&\le & K_3(\|y-x\|^2 +|T(x)-\beta|^2).
\end{eqnarray*} 
Let us now consider $(II)$. We have, for suitable constants $K_4,K_5 >0$,
\begin{eqnarray*}
  (II) &\le & \|\lambda(T(x)) - \lambda(\beta)\|\, \|y-x_1\| \\
  &\le & K_4|T(x) - \beta|(\|y-x\| + K|T(x) - \beta|)\\
  &\le & K_5(\|y-x\|^2 + |T(x) - \beta|^2).
\end{eqnarray*}
Finally, we have, for a suitable constant $K_6 >0$,
\begin{eqnarray*}
 (III) &=& \int_{0}^{T(x)-\beta}\! \langle  \lambda(T(x)), f(\bar{x}(s),\bar{u}(s)\rangle ds\\
 &=&\int_{0}^{T(x)-\beta}\!\! \langle \lambda(T(x)-s), f(\bar{x}(s),\bar{u}(s)\rangle ds + \int_{0}^{T(x)-\beta}\!\!\langle 
\lambda(T(x)) - \lambda(T(x)-s), f(\bar{x}(s),\bar{u}(s)\rangle ds\\
 &&\text{(since the minimized Hamiltonian is constant)}\\
 &\le & (T(x)-\beta)h(x,\zeta) +  \int_{\beta}^{T(x)}{K_6|T(x) -s|ds}\\
 &= & (T(x)-\beta)h(x,\zeta) +  \frac{K_6}{2}|T(x) -\beta|^2.
\end{eqnarray*}
Putting the estimates together, we obtain
\[
\langle\zeta, y -x\rangle \le (T(x) - \beta)h(x,\zeta) + K_7(\|y-x\|^2 + |T(x) - \beta|^2),
\]
for a suitable positive constant $K_7$. The proof of (\ref{norepi}) is concluded in the case (i).
  
We are now going to consider the case (ii). Since $\vartheta \le 0$, it is enough to prove (\ref{norepi}) for $\beta = T(y)$.\\
Since $x$ is an optimal point, there exists $x_1$ such that $T(x_1)=T(y)$ together with an optimal
pair, still denoted $\bar{x}(\cdot)$ and $\bar{u}(\cdot)$, such that $\bar{x}(T(y)-T(x))=x$. Let $\lambda(\cdot)$ denote the extension up to the time $T(y)$
of the solution of \eqref{adj}. Since the  optimal control is unique and
bang-bang with finitely many switchings, it is easy to prove that $h(\bar{x}(T(y)-t),\lambda(t)) =
\langle \lambda(t), f(\bar{x}(T(y)-t),\bar{u}(T(y)-t)) \rangle=
\text{constant}$ for all $t \in [0,T(y)]$. Then by using the same argument of the case (i), 
one can easily show that (\ref{norepi}) holds true.
The proof is complete.
\end{proof}

\subsection{Minimized Hamiltonian and non-Lipschitz points}\label{sec:minham}
This section is devoted to identify points around which the minimum time function $T$ is not Lipschitz as points where the proximal 
normal cone to epi$(T)$ contains a horizontal vector $\zeta \ne 0$. It will also turn out that if $x$ is a non-Lipschitz point and 
$\zeta$ is such a vector, then $h(x,\zeta) = 0$. A kind of converse statement can also be proved.

All results are valid in the domain where epi$(T)$ has positive reach. So, for linear systems, they hold globally, while for nonlinear 
two dimensional systems of the type (\ref{nonlin}), they hold in a neighborhood of the origin whose size can be estimated and depends only on the data.

\begin{Definition} \label{definonlip} We say that a function $T: \R^N \rightarrow \R$ is non-Lipschitz at $x$ provided there exist 
two sequences $\{x_i\}, \{y_i\}$ such that $x_i \ne y_i$ for all $i$, $\{x_i\},\{y_i\}$ converge to $x$ and
$$ \limsup_{i\to\infty} {\frac{|T(y_i) - T(x_i)|}{\|y_i - x_i\|}} = +\infty .$$
\end{Definition}
Observe that the set of non-Lipschitz points is closed.\\
The first result does not require $T$ to be a minimum time function.
\begin{Proposition}\label{Nonlip}
Let $\Omega \subset \R^N$ be open and let $T $ be continuous in $\Omega$ and such that $epi(T)$ has locally positive reach.
Let $\bar{x }\in \Omega$. Then $T$ is non-Lipschitz at $\bar{x}$ if and only if there exists a nonzero vector $\zeta \in \R^N$ such that
\[
(\zeta,0) \in N_{\mathrm{epi}(T)}(\bar{x},T(\bar{x})).
\]
\end{Proposition}
\begin{proof}
By Theorem 9.13 in \cite{RW}, $T$ is non-Lipschitz at $\bar{x}$ if and only if $\partial ^{\infty}T(\bar{x})$ contains a nonzero vector $\zeta$. 
This condition is equivalent to $(\zeta,0) \in N_{\mathrm{epi}(T)}(\bar{x},T(\bar{x}))$.
\end{proof}
Now we restrict ourselves to the case where $T$ is the minimum time function to reach the origin for (\ref{lin}) or for (\ref{nonlin}). 
We assume that the conditions ensuring that epi$(T)$ has positive reach are satisfied (see \cite[Theorem 3.7]{CMW} and
\cite[Theorem 6.5]{GK}).
\begin{Proposition}\label{Hamzero}
Let $T$ denote the minimum time function to reach the origin for (\ref{lin}) or for (\ref{nonlin}). Let $\bar{x} \ne 0$ and $\delta > 0$ 
be such that the epigraph of $T$ restricted to $\bar{B}(\bar{x},\delta)$ has positive reach. Let $\zeta \in \R^N$, $\zeta \neq 0$. Then
\[
\zeta \in \partial ^{\infty}T(\bar{x}) \,\,\, \text{if and only if}\,\,\, h(\bar{x},\zeta) = 0 \,\,\text{and}\,\, 
\zeta \in N_{\mathcal{R}_{T(\bar{x})}}(\bar{x}).
\]
\end{Proposition} 
\begin{proof}
Recalling Proposition \ref{Nonlip}, $\zeta \in \partial ^{\infty}T(\bar{x})$ if and only if $(\zeta,0) \in N_{\mathrm{epi}(T)}(\bar{x},T(\bar{x}))$, 
i.e., for a suitable constant $c\ge 0$,
\begin{equation} \label{3.3}
\langle \zeta, y - \bar{x}\rangle \le c\|\zeta\| \left(\|y-\bar{x}\|^2 +|\beta - T(\bar{x})|^2\right),
\end{equation}
for all $y\in \bar{B}(\bar{x},\delta)$ and for all $\beta \ge T(y)$.

Let $\zeta\in \partial ^{\infty}T(\bar{x})$. If $y\in \mathcal{R}_{T(\bar{x})}$, then we can take $\beta = T(\bar{x})$ in (\ref{3.3}) and so
$$
\langle \zeta, y - \bar{x}\rangle \le c\|\zeta\|\cdot\|y-\bar{x}\|^2,
$$
i.e., $\zeta \in N_{\mathcal{R}_{T(\bar{x})}}(\bar{x})$.

Recalling (i) in Proposition \ref{generalhamilt}, if $\zeta \in N_{\mathcal{R}_{T(\bar{x})}}(\bar{x})$ then $h(\bar{x},\zeta) \le 0$. 
Assume by contradiction that $h(\bar{x},\zeta) <0$. Then, by using (ii) in Proposition \ref{generalhamilt} there exists $\alpha > 0$ such that
\[
 (\alpha \zeta, -1) \in N_{\mathrm{epi}(T)}(\bar{x},T(\bar{x})).
\]
Since $N_{\mathrm{epi}(T)}(\bar{x},T(\bar{x}))$ is convex, we have, for any $\lambda \in (0,1)$,
\begin{eqnarray*}
v_{\lambda}&:=& \lambda (\zeta,0) + (1-\lambda)(\alpha\zeta, -1)\\
&=& (\lambda\zeta + (1-\lambda)\alpha\zeta, \lambda -1) \in N_{\mathrm{epi}(T)}(\bar{x},T(\bar{x})).
\end{eqnarray*}
This implies
$$ \frac{v_\lambda}{1-\lambda} = \left(\frac{\lambda\zeta + (1-\lambda)\alpha\zeta}{1-\lambda},-1\right) \in N_{\mathrm{epi}(T)}(\bar{x},T(\bar{x})),
$$
i.e., 
$$ \frac{v_\lambda}{1-\lambda} \in \partial T(\bar{x}).
$$
By Theorem 5.1(b) in \cite{WYu}, we have 
$$h\left(\bar{x},\frac{\lambda\zeta + (1-\lambda)\alpha\zeta}{1-\lambda} \right) = -1, $$
i.e., $h(\bar{x},\lambda\zeta + (1-\lambda)\alpha\zeta) = \lambda-1$, for all $\lambda \in (0,1)$.
Letting $\lambda \to 1^{-}$ in the above equality, we obtain $h(\bar{x},\zeta) = 0$, which is a contradiction. Thus $h(\bar{x},\zeta) = 0$.

Conversely, let $\zeta \in N_{\mathcal{R}_{T(\bar{x})}} (\bar{x})$ be a nonzero vector such that $h(\bar{x},\zeta) = 0$. 
Applying Proposition \ref{generalhamilt}, we see that $(\zeta,0) \in N_{\mathrm{epi}(T)}(\bar{x},T(\bar{x}))$, which says 
exactly that $\zeta \in \partial^{\infty}T(\bar{x})$. The proof is concluded.

\end{proof}

\section{Linear systems in $\R^N$}\label{sec:lin}
This section is devoted to the study of non-Lipschitz points for the minimum time function to reach the origin for a linear 
autonomous system of the form:
\begin{equation}\label{linsys}
\dot{x} = Ax + Bu
\end{equation}
where $A\in \mathbb{M}_{N\times N}$ and $B\in \mathbb{M}_{N\times M}$, $1 \le M \le N$, and $u \in [-1,1]^M$.

We assume that \eqref{linsys} is normal, i.e., for each column $b_i, i = 1,\dots,M$, of $B$, the Kalman rank condition
\begin{equation} \label{kalman}
\mathrm{rank}[b_i,Ab_i,\dots,A^{N-1}b_i] = N
\end{equation}
holds. We set also
\begin{equation} \label{rkB}
k = \mathrm{rank}B.
\end{equation}
Of course, $1 \le k\le M$. 

We will first characterize the set $\mathcal{S}$ of non-Lipschitz points of $T$ as
\begin{equation}\label{defS}
\begin{split}
\mathcal{S} &= \Big\{
x\in \R^N: \,\text{there exist}\, r>0\,\text{and}\, \zeta \in \mathbb{S}^{N-1} \text{ such that }\\
&\qquad x = \sum_{i=1}^M \int_{0}^{r}e^{A(t-r)}b_i\,\mathrm{sign}\left(\langle \zeta,e^{At}b_i\rangle\right)dt \,\text{ and }\,
\langle \zeta,b_i\rangle = 0\;\forall i = 1,\dots, M
\Big\}.
\end{split}
\end{equation}
If $k=N$, then $\mathcal{S}$ is empty. If $k <N$, $\mathcal{S}$ is nonempty and
we will prove also that $\mathcal{S}$ is $(N-k)$-rectifiable, with positive and locally finite
$\mathcal{H}^{N-k}$-measure. The positivity part will be contained in Section \ref{sec:propag}.

>From now on, we assume
\begin{equation}\label{condk}
k<N.
\end{equation}
\begin{Rem}\label{remnewS}
It is easy to see that
\begin{equation}\label{NewS}
\begin{split}
\mathcal{S} &= \Big\{
x\in \R^N: \text{there exist }\, r >0\,\text{ and }\, \zeta \in \mathbb{S}^{N-1}\,\text{ such that }\\
&\qquad\qquad x = \sum_{i=1}^M \int_{0}^{r}e^{-At}b_i\,\mathrm{sign}\left(\langle \zeta,e^{-At}b_i\rangle\right)dt,\\
&\qquad\qquad\qquad\quad \zeta \in N_{\mathcal{R}_r}(x)
\,\text{ and }\,\langle \zeta,e^{-Ar}b_i\rangle = 0\;\;\forall i = 1,\dots, M
\Big\}.
\end{split}
\end{equation}
%In other words, recalling the Maximum Principle (see Theorem \ref{th:PMPlin} and (iii) in Theorem \ref{th:PMPnonlin}) and %\eqref{equivlin},
%$\mathcal{S}$ is the set of points
%which are reachable from the origin by a trajectory of the reversed dynamics which is both time optimal and (according to %the next lemma) has vanishing Hamiltonian.
\end{Rem}
\noindent We state first a technical lemma concerning an explicit computation of the minimized Hamiltonian.
Before stating it, let us observe that condition \eqref{kalman} implies that the function
$t\mapsto\langle\bar{\zeta},e^{-At}b_i\rangle$ is not identically $0$ for all $i=1,\dots,M$.
\begin{Lemma}\label{lemmaham}
Let $r>0$, $\bar{x}\in \R^N$ and $\bar{\zeta}\in \mathbb{S}^{N-1}$ be such that
\[
\bar{x} = \sum_{i=1}^M\int_{0}^{r}{e^{-At}b_i\,\mathrm{sign}\left(\langle\bar{\zeta},e^{-At}b_i\rangle\right)dt}.
\]
Then
\begin{equation} \label{compham}
h(\bar{x},\bar{\zeta}) = -\sum_{i=1}^{M} \Big| \langle\bar{\zeta},e^{-Ar}b_i\rangle \Big|.
\end{equation}
\end{Lemma} 
\begin{proof}
We have
\begin{eqnarray*}
h(\bar{x},\bar{\zeta}) &=& \langle\bar{\zeta},A\bar{x}\rangle + \min_{u\in [-1,1]^M}\langle\bar{\zeta},Bu\rangle\\
&=&\langle\bar{\zeta},A\bar{x}\rangle + \min_{\begin{subarray}{c}|u_i|\le 1\\i=1,...,M\end{subarray}}\sum_{i=1}^M\langle\bar{\zeta},b_iu_i\rangle\\
&=& \Big\langle \bar{\zeta},\sum_{i=1}^M\int_{0}^{r}{Ae^{-At}b_i\,\mathrm{sign}\left(\langle\bar{\zeta},e^{-At}b_i\rangle\right)dt}\Big\rangle 
-\sum_{i=1}^M|\langle\bar{\zeta},b_i\rangle|\\
&=& \sum_{i=1}^M  \left( \int_{0}^{r}{\big\langle\bar{\zeta},Ae^{-At}b_i\,\mathrm{sign}\left(\langle\bar{\zeta},e^{-At}b_i\rangle\right)\big\rangle dt}-
|\langle\bar{\zeta},b_i\rangle |\right)\\
&=:& \sum_{i=1}^Mh_i(\bar{x},\bar{\zeta}).
\end{eqnarray*}
Set $g_i(t) :=\langle\bar{\zeta},e^{-At}b_i\rangle$, $t\ge 0$. Then being not identically zero, $g_i$ vanishes at most finitely 
many times in $[0,r]$, say at $0\le t_1 < \dots < t_k \le r$. We have, for $i=1,\dots,M$,
\begin{eqnarray*}
h_i(\bar{x},\bar{\zeta}) &=&\int_{0}^{r}{-\dot{g}_i(t)\mathrm{sign}(g_i(t))dt}-|\langle\bar{\zeta},b_i\rangle|\\
&=& -\int_{0}^{t_1}{\dot{g}_i(t)\,\mathrm{sign}(g_i(t))dt} - \sum_{j=1}^{k-1}\int_{t_j}^{t_{j+1}}{\dot{g}_i(t)\,\mathrm{sign}(g_i(t))dt} \\
&&-\int_{t_k}^{r}{\dot{g}_i(t)\,\mathrm{sign}(g_i(t))dt}-|\langle\bar{\zeta},b_i\rangle|\\
&=& \left(g_i(0) - g_i(t_1) \right)\mathrm{sign}\left(g_i\left(\frac{t_1}{2}\right)\right) +\sum_{j=1}^{k-1} \left(g_i(t_j) - g_i(t_{j+1}) \right)
\mathrm{sign}\left(g_i\left(\frac{t_j+t_{j+1}}{2}\right)\right)\\
&& + \left(g_i(t_k) - g_i(r) \right)\mathrm{sign}\left(g_i\left(\frac{t_k+r}{2}\right)\right)-|\langle\bar{\zeta},b_i\rangle|\\
&=&g_i(0)\mathrm{sign}\left(g_i\left(\frac{t_1}{2}\right)\right) -g_i(r)\mathrm{sign}\left(g_i\left(\frac{t_k+r}{2}\right)\right)-g_i(0)\mathrm{sign}(g_i(0)).
\end{eqnarray*}
If $g_i(0) \ne 0$ and $g_i(r) \ne 0$, then $\mathrm{sign}(g_i(0))=  \mathrm{sign}\left(g_i\left(\frac{t_1}{2}\right)\right)$ and
$\mathrm{sign}(g_i(r)) = 
\mathrm{sign}\left(g_i\left(\frac{t_k+r}{2}\right)\right)$. Thus
$h_i(\bar{x},\bar{\zeta}) = -|g_i(r)| =-|\langle\bar{\zeta},e^{-Ar}b_i\rangle|$.
Analogously, if $g_i(0) = 0$ and $g_i(r) = 0$, then $h_i(\bar{x},\bar{\zeta}) = 0 = -|g_i(r)|$. If $g_i(0) \ne 0$ and $g_i(r) = 0$,
then $h_i(\bar{x},\bar{\zeta}) = -|g_i(r)|$. Finally, if $g_i(0) = 0$ and $g_i(r) \ne 0$, then $h_i(\bar{x},\bar{\zeta}) = -|g_i(r)|.$
In all cases, we have
$$h_i(\bar{x},\bar{\zeta}) = -|g_i(r)| =-|\langle\bar{\zeta},e^{-Ar}b_i\rangle|,$$
and (\ref{compham}) follows.
\end{proof}
\begin{Rem}\label{Snonlip}
The characterization \eqref{NewS}, thanks to Lemma \ref{lemmaham} and Theorem \ref{th:PMPlin}, implies also that
\[
\mathcal{S}=\{x \in \R^N: \exists r>0\,\text{ such that }\,x\in \mathrm{bdry}\mathcal{R}_r\,\text{ and }\,\zeta \in \mathbb{S}^{N-1} 
\cap N_{\mathcal{R}_r}(x)\,\text{ for which }\,h(x,\zeta) = 0\}. 
\] 
\end{Rem}
The computation of the Hamiltonian contained in \eqref{compham} permits to prove the following characterization of
non-Lispchitz points of $T$. We recall that, under the assumption \eqref{condk}, the set $\mathcal{S}$ is nonempty.
\begin{Theorem}\label{charnonlip} 
Let $\bar{x} \in \R^N \setminus \{ 0\}$. Then $T$ is non-Lipschitz at $\bar{x}$ if and only if $\bar{x}\in \mathcal{S}$. Moreover, 
$\mathcal{S}$ is invariant for optimal trajectories of the reversed dynamics having vanishing Hamiltonian.
\end{Theorem}
\begin{proof}
Let $\bar{x}\ne 0$ be a non-Lipschitz point of $T$ and set $r = T(\bar{x}) >0$. We recall that by Theorem 3.7 in \cite{CMW} epi($T$) 
has positive reach. Therefore, by Propositions \ref{Nonlip} and \ref{Hamzero}, there exists  $\bar{\zeta} \in \mathbb{S}^{N-1} \cap N_{\mathcal{R}_r}(\bar{x})$ 
such that $h(\bar{x},\bar{\zeta})=0$.
 
Let $\bar{u}(\cdot)$ be the optimal control steering $\bar{x}$ to the origin in the minimum time $r$. 
Then $\tilde{u}(t) = \bar{u}(r-t), 0\le t\le r$,
steers the origin to $\bar{x}$ in the optimal time $r$ for the reversed dynamics $\dot{x} = -Ax - Bu, u\in [-1,1]^M$, namely
$$ \bar{x} = -\int_{0}^{r} {e^{-A(r-t)}B\tilde{u}(t)dt} =-\sum_{i=1}^M\int_{0}^{r} {e^{-A(r-t)}b_i\tilde{u}_i(t)dt},  $$
where $b_i$ are the columns of $B$ and $\tilde{u}=(\tilde{u}_1,\dots,\tilde{u}_M)$.\\
By the Maximum Principle,
$$\tilde{u}_i(t) = -\mathrm{sign}\left( \langle \bar{\zeta},e^{-A(r-t)}b_i\rangle\right),\,\,\,i=1,\dots,M.
$$
Therefore,
\begin{eqnarray}\label{4.11}
\bar{x}&=&\sum_{i=1}^M\int_{0}^{r} {e^{-A(r-t)}b_i\,\mathrm{sign}\left( \langle \bar{\zeta},e^{-A(r-t)}b_i\rangle\right)dt} \nonumber\\
&=& \sum_{i=1}^M\int_{0}^{r} {e^{-At}b_i\,\mathrm{sign}\left( \langle \bar{\zeta},e^{-At}b_i\rangle\right)dt}.
\end{eqnarray}
Since $ h(\bar{x},\bar{\zeta}) = 0$, (\ref{compham}) yields
\[
\langle\bar{\zeta}, e^{-Ar}b_i\rangle = 0,\,\,\,\,i = 1,\dots,M,
\]
and the proof of $\bar{x} \in \mathcal{S}$ is concluded, recalling \eqref{4.11} and \eqref{NewS}.
   
Conversely, let $\bar{x} \in \mathcal{S}$ and set $r = T(\bar{x})>0$. Then by \eqref{NewS} there exists $\bar{\zeta} \in \mathbb{S}^{N-1} 
\cap N_{\mathcal{R}_r}(\bar{x})$ 
such that $\langle\bar{\zeta},e^{-Ar}b_i\rangle = 0$ for all $i = 1,\dots,M$, and
$$
\bar{x} = \sum_{i=1}^M\int_{0}^{r}{e^{-At}b_i\,\mathrm{sign}\left(\langle\zeta,e^{-At}b_i\rangle\right)dt}.
$$
Recalling (\ref{compham}), $h(\bar{x},\bar{\zeta)}) = 0$. Therefore, by Propositions \ref{Nonlip} and \ref{Hamzero},
$T$ is non-Lipschitz at $\bar{x}$.

The last statement is an immediate consequence of Remark \ref{remnewS}. In fact, the alternative espression of $\mathcal{S}$ given in \eqref{NewS},
together with the Maximum Principle (see Theorem \ref{th:PMPlin})
shows that every $x\in\mathcal{S}$ is the endpoint of a time optimal trajectory for the reversed dynamics with vanishing Hamiltonian, starting from the origin.
The proof is concluded.
\end{proof} 
\noindent We prove now a rectifiability property for $\mathcal{S}$, which is the main result of this section.
\begin{Theorem}\label{N-1rect}
Let $\mathcal{S}$ be defined according to (\ref{defS}) and let $k$ be given by (\ref{rkB}).
Then $\mathcal{S}$ is closed and countably $(N-k)$-rectifiable. More precisely, for every $r>0$ there exist countably many Lipschitz functions
$f_j:\mathbb{R}^{N-k-1}\to \mathbb{R}^N$ such that $\mathcal{S}\cap\mathrm{bdry}\, \mathcal{R}_r\subseteq \cup_j f_j(\mathbb{R}^{N-k-1})$.
\end{Theorem}
\begin{proof}
The characterization of $\mathcal{S}$ contained in Remark \ref{Snonlip} implies immediately its closedness.

We now deal with the rectifiability of $\mathcal{S}$.
Observe that the set
\begin{equation}\label{defzeta}
Z := \{\zeta \in \mathbb{S}^{N-1}: \langle \zeta,b_i\rangle = 0,\,\forall i=1,\dots,M\}
\end{equation}
is a $N-(k+1)$ manifold. We define, for $i=1,\dots,M$,
\[
g^+_i(t,\zeta) = \langle \zeta,e^{At}b_i\rangle,\qquad\zeta \in Z,\; t\ge 0,
\]
\[ 
\Phi_i(r,\zeta) = \int_{0}^{r}{e^{A(t-r)}b_i\,\mathrm{sign}(g^+_i(t,\zeta))dt}
\]
and
\begin{equation}\label{defSM}
\Sigma_i = \big\{x\in \R^N: \text{ there exist }r >0\text{ and }\zeta \in Z\text{ such that } x = \Phi_i(r,\zeta)\big\}.
\end{equation}
We claim now that 
\begin{equation} \label{claimsigma}
\text{each } \Sigma_i \text{ is contained in a countable union of Lipschitz graphs of $N-k$ variables.}
\end{equation}
To this aim, we fix the index $i$ and drop the corresponding subscript for the sake of simplicity.\\
We set the following definitions. Fix $\tau >0$. For every $(N-1)$-tuple of nonnegative integers,
\[
\mathbf{j} = (j_1,\dots,j_{N-1}) \in \mathbb{N}^{N-1},
\]
we define
\begin{eqnarray*}
Z_{\mathbf{j}} = \Big\{\zeta\in Z: &&\text{$g^+(t,\zeta)$ has in the interval $[0,\tau]$ exactly}\\
&&\qquad\text{$j_1$ zeros of multiplicity}\,\,1,\\ 
&&\qquad \dots\\
&&\qquad\text{$j_{N-1}$ zeros of multiplicity}\,\,N-1\Big\}.
\end{eqnarray*}
We set also $|\mathbf{j}| = j_1 + \dots + j_{N-1}$ and observe that, thanks to \eqref{defzeta}, we can consider only $\mathbf{j}'s$ such that
$|\mathbf{j}| \ge 1$. Moreover, for any positive integer $d$ and $\mathbf{j}\in\mathbb{N}^{N-1}$ with $|\mathbf{j}| > 1$ we define
\[
Z^d_{\mathbf{j}} = \left\{\zeta\in Z_{\mathbf{j}}: \min \{|\tau_1 - \tau_2|: g^+(\tau_1,\zeta) = g^+(\tau_2,\zeta) = 0, \tau_1 \ne \tau_2\}\ge \frac{1}{d} \right\}. 
\]
Invoking Lemma 3.2 in \cite{GK}, we obtain that
\[
Z =\Big(\bigcup_{|\mathbf{j}|=1}Z_{\mathbf{j}}\Big)\; \cup\; \Bigg(\bigcup_{d=1}^{\infty}
\bigcup_{\scriptstyle \begin{array}{c}\mathbf{j}\in \mathbb{N}^{N-1}\\ |\mathbf{j}|>1\end{array}}
         Z^d_{\mathbf{j}}\Bigg).
\]
We define finally the map
\begin{equation}  \label{defY}
\begin{split}
Y: &\; Z\rightarrow L^1(0,\tau)\\
&\; \zeta \mapsto \mathrm{sign}(g^+(\cdot,\zeta)),
\end{split}
\end{equation}
and, for all $\mathbf{j}\in \mathbb{N}^{N-1}$, the sets
\[
Z_{\mathbf{j}}^{d,\pm}=\Big\{ \zeta \in Z_{\mathbf{j}}^d: \lim_{t\to 0+}\mathrm{sign}(g^+(t,\zeta)) = \pm 1 \Big\}.
\]
We fix now $\mathbf{j} \in \mathbb{N}^{N-1}$.
If $|\mathbf{j}|=1$, then $Y(\zeta)(t) \equiv \pm 1$ for all $\zeta \in Z_{\mathbf{j}}$, $t\in (0,\tau]$, and so $Y$ is locally Lipschitz in $Z_{\mathbf{j}}$.
We claim that $Y$ is locally Lipschitz also in $Z_{\mathbf{j}}^{d,+}$ and in $Z_{\mathbf{j}}^{d,-}$ for each $\mathbf{j} \in \mathbb{N}^{N-1}$. 
The argument for $Z_{\mathbf{j}}^{d,+}$ and $Z_{\mathbf{j}}^{d,-}$ is the same, so we perform it only for $Z_{\mathbf{j}}^{d,+}$, $|\mathbf{j}| >1$, $d\ge 1$.

So, fix $|\mathbf{j}| >1$, $d\ge 1$ and $\zeta_{0} \in Z_{\mathbf{j}}^{d,+}$. Let $t_1,\dots,t_{|\mathbf{j}|}$ be the zeros of 
$g^+(\cdot,\zeta_{0})$ in $[0,\tau]$, each one with multiplicity $m_h$, $h= 1,\dots, |\mathbf{j}|$.\\
By continuity and the implicit function theorem, for each $h= 1,\dots, |\mathbf{j}|$ there exist a compact neighborhood $V_h$ of $\zeta_0$,
a neighborhood $I_h$ of $t_h$ and a $\mathcal{C}^1$-function $\varphi_h: V_h \rightarrow I_h$ such that 
\begin{equation}\label{dernonzero}
\frac{\partial^{m_h}}{\partial t^{m_h}}g^+(t,\zeta)\neq 0\qquad\forall (t,\zeta)\in I_h\times V_h 
\end{equation}
and
\begin{equation}\label{intIi}
\left\{ (\zeta,t) \in V_h\times I_h: \frac{\partial^{m_h-1}}{\partial t^{m_h-1}}g^+(t,\zeta) = 0 \right\} = \mathrm{graph}(\varphi_h). 
\end{equation}
The neighborhoods $I_h$ can be taken disjoint and satisfying $|I_h| \le \frac{1}{2d}$.
We choose now $V=V(\zeta_0)\subseteq \cap_{h=1}^{|\mathbf{j}|}V_h$ with the further requirement that  for all $\zeta \in V$,
the set $\{t\in [0,\tau]: g^+(t,\zeta) = 0\}$ is contained in $\bigcup_{h=1}^{|\mathbf{j}|} I_h$.
Since $|I_h| \le \frac{1}{2d}$, the function $g^+(t,\zeta)$ has at most one zero in each $I_h$.

Set $V_{\mathbf{j}}(\zeta_0) = V\cap Z_{\mathbf{j}}^{d,+}$. Without loss of generality, we may assume that $Z_{\mathbf{j}}^{d,+}$ 
is contained in a finite union of such $V_{\mathbf{j}}(\cdot)$, say $Z_{\mathbf{j}}^{d,+} = \bigcup_\ell V_{\mathbf{j}}(\zeta_\ell)$. 
We write the functions corresponding to $V_{\mathbf{j}}(\zeta_\ell)$ as $\varphi_{h}^{\ell}(\zeta), h = 1,\dots, |\mathbf{j}|$, and observe
that each $\varphi_{h}^{\ell}$ is Lipschitz continuous on $V_{\mathbf{j}}(\zeta_\ell)$, say with Lipschitz constant $L_{h}^{\ell}$.
We denote also the intervals corresponding to $V_{\mathbf{j}}(\zeta_\ell)$ as $I_h^\ell$, $h=1,\ldots,|\mathbf{j}|$. Of course, some of the
$V_{\mathbf{j}}(\zeta_\ell)$'s may be the singleton $\{\zeta_\ell\}$, and in this case everything trivializes. Fix now an index $\ell$.

We claim that, 
for each $\zeta \in V_{\mathbf{j}}(\zeta_\ell)$, $g^+(\cdot,\zeta)$ has a zero of multiplicity $m_h$ exactly at 
$\varphi_{h}^{\ell}(\zeta)$, $h = 1,\dots, |\mathbf{j}|$, and does not have other zeros in $[0,\tau]$. Indeed, by construction
for each $\zeta \in V_{\mathbf{j}}(\zeta_\ell)$ all zeros of $g^+(\cdot,\zeta)$ are contained in $\bigcup_{h=1}^{|\mathbf{j}|} I_h^\ell$.
Let $\kappa$ be the largest index $k$ such that $j_k\neq 0$. Again by construction, for each $\zeta\in V_{\mathbf{j}}(\zeta_\ell)$
the map $t\mapsto g^+(t,\zeta)$ has exactly $j_{\kappa}$ zeros of multiplicity $\kappa$. Moreover, such $j_{\kappa}$ zeros must belong
to the same intervals $I_h^\ell$ to which the 
$j_{\kappa}$ zeros of multiplicity $\kappa$ of $g^+(\cdot,\zeta_\ell)$ belong, since in all other intervals we have
at least one nonvanishing derivative of order $\le\kappa -1$. Owing to \eqref{intIi} with $m_h=\kappa$, such zeros must occur at $\varphi_h^{\ell}(\zeta)$,
for the corresponding index $h$.
Let now $\kappa_1$ be the largest positive integer $<\kappa$ such that $j_{\kappa_1}>0$. By definition of $V_{\mathbf{j}}(\zeta_\ell)$, for each
$\zeta\in V_{\mathbf{j}}(\zeta_\ell)$ the map $t\mapsto g^+(\cdot,\zeta)$ does not have zeros of order $k$, with $\kappa_1 < k <\kappa$ and must have
exactly $j_{\kappa_1}>0$ zeros of multiplicity $\kappa_1$. Such zeros cannot belong to the intervals to which the $\kappa$-zeros
of $g^+(\cdot,\zeta)$ belong, since such intervals already contain a zero; on the other hand, by \eqref{dernonzero} they must
belong to the same intervals $I_h^\ell$ to which the zeros of multiplicity $\kappa_1$ of $g^+(\cdot,\zeta_\ell)$ belong,
and therefore they must occur at $\varphi_h^{\ell}(\zeta)$, for the corresponding index $h$. An analogous argument can be performed for all further
indexes $k < \kappa_1$ such that $j_k\neq 0$. Therefore the claim is proved.

We are now ready to show that $Y$ is Lipschitz on $V_{\mathbf{j}}(\zeta_\ell)$.
Indeed, fix the index $\ell$ and let $\zeta_1,\zeta_2 \in V_{\mathbf{j}}(\zeta_\ell)$.
Then
\vspace{-0.283truecm}
\begin{align*}
\|Y(\zeta_2) - Y(\zeta_1) \|_{L^1(0,\tau)} &\; \le\;  2 \sum_{\scriptscriptstyle\begin{array}{c}h=1\\ m_h\text{ is odd}\end{array}}^{|\mathbf{j}|}|\varphi_h^\ell(\zeta_2)-\varphi_h^\ell(\zeta_1)|\\
&\;\le\; 2 \sum_{\scriptscriptstyle\begin{array}{c}h=1\\ m_h\text{ is odd}\end{array}}^{|\mathbf{j}|}L_h^\ell\|\zeta_2-\zeta_1\|,
\end{align*}
which proves the claim.

The Lipschitz continuity of $Y$ on each $V_{\mathbf{j}}(\zeta_\ell)$ implies immediately that, for all fixed $r\in [0,\tau]$, 
the function $\zeta \mapsto \Phi(r,\zeta)$ is Lipschitz in the same set.
On the other hand, the function $r \mapsto \Phi(r,\zeta)$ is immediately seen to be Lipschitz on $[0,\tau]$.
Consequently, the set $\Sigma$ defined in (\ref{defSM}) is contained in a countable union of Lipschitz graphs of $N-k$ variables.
The $(N-k)$-rectifiability of $\mathcal{S}$ now follows easily and the proof is concluded.
\end{proof}

\section{Nonlinear systems in $\R^2$}\label{sec:nonlin}
This section is devoted to the study of non-Lipschitz points of $T$ for the nonlinear system
\begin{equation} \label{nonlinsys}
\begin{cases}
\dot{x}(t)=F(x(t)) + G(x(t)) u(t),\\
u(t) \in [-1,1]^M,\\
x(0)=x,
\end{cases}
\end{equation}
where the state $x$ is in $\R^2$ and $M$ is either 1 or 2.

The assumptions are the following: 
\begin{itemize}
\item[1)] $F:\R^2 \to \R^2$ and $G:\R^2 \to \mathbb{M}_{2\times M}$ are of class $\mathcal{C}^{1,1}$ and all partial derivatives are Lipschitz with constant $L$;
\item[2)] $F(0)=0$;
\item[3)] $\mathrm{rank}[G_i(0),DF(0)G_i(0)]=2$\, for\, $i=1,\ldots ,M$, where we mean $G=G_1$ if $M=1$ and $G = (G_1,G_2)$ if $M=2$;
\item[4)] $DG(0)=0$.
\end{itemize}
Theorems 5.1, 6.2 and 6.5 in \cite{GK} yield that there exists $\mathcal{T} >0$, depending only on $L$, $DF(0)$, and $G(0)$,
such that for all $0< \tau <\mathcal{T}$,
\begin{itemize}
\item[a)] $\mathcal{R}_\tau$ is strictly convex and for all $x\in \mathrm{bdry}\mathcal{R}_{\tau}$ there exists a unique
optimal control $u(\cdot)$ steering $x$ 
to the origin in the minimum time $\tau$, and $u(\cdot)$ is bang-bang with finitely many switchings,
\item[b)] every $x\in \mathcal{R}_{\tau}$ is optimal (the definition of optimal point was recalled in Section \ref{sect:normals}),
\item[c)] epi($T$) has locally positive reach.
\end{itemize}
We prove here a result which is the nonlinear two dimensional analogue of Theorem \ref{N-1rect}. Fix $0<\tau < \mathcal{T}$ and define
\begin{equation}\label{defSnonlin}
S = \left\{ x\in \mathcal{R}_{\tau}: \exists \zeta \in \mathbb{S}^1 \cap N_{\mathcal{R}_{T(x)}}(x)\,\text{ such that }\,h(x,\zeta) = 0\right\}.
\end{equation}
Recalling Propositions \ref{Nonlip} and \ref{Hamzero}, $S$ is exactly the set of non-Lipschitz points of $T$ within $\mathcal{R}_{\tau}$.
We show first that $S$ is invariant for a class of optimal trajectories and then that it is countably $\mathcal{H}^1$-rectifiable.
\begin{Proposition}\label{Sinv}
Let $S$ be defined according to \eqref{defSnonlin} and let $F$, $G$ satisfy the assumptions 1) -- 4). Then
$S$ is invariant for optimal trajectories.
\end{Proposition}
\begin{proof}
We wish to prove that if $\bar{x} \in S$ and $x(\cdot)$ is the optimal 
trajectory steering $\bar{x}$ to the origin in the minimum time $T(\bar{x})$ then $x(t) \in S$ for all $0\le t\le T(\bar{x})$. 
In fact, let $\bar{u}(\cdot)$ be the corresponding optimal control and set $\tilde{u}(t) = \bar{u}(T(\bar{x})-t), 0\le t \le T(\bar{x})$. 
Let $\tilde{x}(\cdot)$ be the solution of the system
\begin{equation}\label{revdyn}
\left\{\begin{array}{ll}
\dot{x}(t)\: = \: -F(x(t))-G(x(t))\tilde{u}(t), \\
x(0)  \: = \:  0
\end{array}\right.
\end{equation}
and let $\bar{\zeta} \in \mathbb{S}^1 \cap N_{\mathcal{R}_{\tau}} (\bar{x})$ be such that $h(\bar{x},\bar{\zeta}) =0$.\\
\textit{Claim.} The solution $\tilde{\lambda}(t)$, $t\in [0,T(\bar{x})]$ of the adjoint system
\begin{equation}\label{adjtilde}
\left\{\begin{array}{ll}
\dot{\lambda}(t)\: = \:\lambda(t)\big( DF(\tilde{x}(t))+DG(\tilde{x}(t))\tilde{u}(t)\big) \\
\lambda(T(\bar{x}))  \: = \:  \bar{\zeta},
\end{array}\right.
\end{equation}
satisfies the following properties:
\begin{itemize}
\item[(i)] $\tilde{u}_i(t) = \mathrm{sign}\big(\langle \tilde{\lambda}(t),-G_i(\tilde{x}(t))\rangle\big)$ for a.e. $t\in [0,T(\bar{x})]$, $i=1,\ldots ,M$,
\item[(ii)] $0 = h(\tilde{x}(t),\tilde{\lambda}(t)) = \langle F(\tilde{x}(t)),\tilde{\lambda}(t)\rangle - \sum_{i=1}^{M}|
\langle G_i(\tilde{x}(t)),\tilde{\lambda}(t)\rangle|$ for all $t\in [0,T(\bar{x})]$,
\item[(iii)] $0 \ne \tilde{\lambda}(t) \in N_{\mathcal{R}_{T(\tilde{x}(t))}}(\tilde{x}(t)),$ for all $t\in [0,T(\bar{x})]$.
\end{itemize}
\textit{Proof of the Claim}. We recall that under our assumptions $\mathcal{R}_{T(\bar{x})}$ is strictly convex. In particular,
$N_{\mathcal{R}_{T(\bar{x})}}(\bar{x})$ is the convex hull of its exposed rays (see \cite[p.163]{ro} and \cite[Corollary 18.7.1, p. 169]{ro}).
Therefore let $\zeta \ne 0$ belonging to an exposed ray of $N_{\mathcal{R}_{T(\bar{x})}}(\bar{x})$. 
Recalling (b) in Proposition \ref{generalhamilt}, there exists $\sigma \le 0$ such that 
$\frac{(\zeta,\sigma)}{\sqrt{\|\zeta\|^2+\sigma^2}}$ belongs to an exposed ray of $N_{\mathrm{epi}(T)}(\bar{x},T(\bar{x}))$.\\
By Theorem 4.9 in \cite{CMJCA}, there exists a sequence $\{x_n\} \subset \mathrm{dom}(DT)$ such that $x_n \to \bar{x}$ and
$$
\lim_{n\to \infty} \frac{(DT(x_n),-1)}{\sqrt{\|DT(x_n)\|^2 +1}} = \frac{(\zeta,\sigma)}{\sqrt{\|\zeta\|^2+\sigma^2}}
$$
Let $u_n=(u_{1,n},\ldots,u_{M,n})$ be the optimal control steering the origin to $x_n$.
Since $N_{\mathcal{R}_{T(x_n)}}(x_n)$ is the half ray $\R^{+}DT(x_n)$, 
for $n$ large enough, then Pontryagin's Maximum Principle yields that
\begin{equation}\label{pontr_n}
u_{i,n}(t) = \mathrm{sign}\left(\langle \lambda_n(t),-G_i(x_n(t))\rangle\right),\,\text{a.e.}\,t\in [0,T(\bar{x})], \; i=1,\ldots ,M,
\end{equation}
where $x_n(\cdot)$ is the solution of
\begin{equation*}
\left\{\begin{array}{ll}
\dot{y}\: = \: -F(y)-G(y)u_n \\
y(0)  \: = \:  0,
\end{array}\right.
\end{equation*}
and $\lambda_n$ is the solution of
\begin{equation*}
\left\{\begin{array}{ll}
\dot{\lambda}(t)\: = \:\lambda(t)\big( DF(x_n(t))+DG(x_n(t))u_n(t)\big), &\,a.e.\\
\lambda(T(x_n))  \: = \:  \zeta_n \in \R^{+}DT(x_n).
\end{array}\right.
\end{equation*}
Since all controls $u_n$ are bang-bang with a finite number of switchings independent of $n$, up to a subsequence we
can assume that $u_n(\cdot)$ (where we have put $u_n(t) \equiv 0$ for $t\in (0,T(\bar{x})-T(x_n))$ if $T(x_n) < T(\bar{x})$)
converges pointwise a.e. to some admissible $u_0:[0,T(\bar{x})] \to [-1,1]^M$. Let $x_0(\cdot)$ be the solution of (\ref{revdyn})
with $u_0$ in place of $\tilde{u}$. Since obviously $x_0(T(\bar{x})) =\bar{x}$, by the uniqueness of the optimal control we have
that $u_0 (t) =\tilde{u}(t)$ a.e. on $[0,T(\bar{x})]$.
Up to another subsequence, we can assume that
$x_n(\cdot)$ converges uniformly to $\tilde{x}(\cdot)$ on $[0,T(\bar{x})]$,
and
$\lambda_n(\cdot)$ converges uniformly to $\lambda(\cdot)$ on $[0,T(\bar{x})]$. Then $\lambda(\cdot)$ is the solution of (\ref{adjtilde}) 
with $\zeta$ in place of $\bar{\zeta}$.
Recalling \eqref{pontr_n}, the above convergence properties imply that
\begin{equation}\label{signdef}
\tilde{u}_i(t) = \mathrm{sign}\left(\langle \lambda(t),-G_i(\tilde{x}(t))\rangle\right),\,\text{a.e.}\,t\in [0,T(\bar{x})], \; i=1,\ldots ,M.
\end{equation}
Let now $\bar{\zeta}_1,\bar{\zeta}_2 \in \mathbb{S}^1$ belong to exposed rays of $N_{\mathcal{R}_{T(\bar{x})}}(\bar{x})$ and let 
$\alpha, \beta \ge 0$ be such that $\bar{\zeta} =\alpha \bar{\zeta}_1 + \beta\bar{\zeta}_2$.
Let $\tilde{\lambda}_1(\cdot)$ (resp., $\tilde{\lambda}_2(\cdot)$) be the solutions of (\ref{adjtilde}) with $\bar{\zeta}_1$ 
(resp., $\bar{\zeta}_2$) in place of $\bar{\zeta}$. By (\ref{signdef}), we have, for a.e. $t\in [0,T(\bar{x})]$, that
\[
\tilde{u}_i(t) = \mathrm{sign}\left(\langle \tilde{\lambda}_1(t),-G_i(\tilde{x}(t))\rangle\right) = \mathrm{sign}
\left(\langle \tilde{\lambda}_2(t),-G_i(\tilde{x}(t))\rangle\right), \; i=1,\ldots ,M.
\]
Therefore, for a.e $t\in [0,T(\bar{x})]$,
\begin{eqnarray*}
\tilde{u}_i(t) &=&\mathrm{sign}\left(\langle \tilde{\alpha\lambda}_1(t)+\beta\lambda_2(t),-G_i(\tilde{x}(t))\rangle\right)\\
&=&\mathrm{sign}\big(\langle \tilde{\lambda}(t),-G_i(\tilde{x}(t))\rangle\big), \; i=1,\ldots ,M,
\end{eqnarray*}
which proves (i).\\
To prove (ii), observe that the fact that $h(\tilde{x}(t),\tilde{\lambda}(t))$ is constant follows in a standard way from 
the maximization property (i) (see, e.g., Corollary 6.4 in \cite{GK}).
Since $h(\tilde{x}(T(\bar{x})),\tilde{\lambda}(T(\bar{x}))) = h(\bar{x},\bar{\zeta}) = 0$, (ii) is proved.\\
Statement (iii) again follows from the maximization property (i) (see, e.g., Remark 5.2 in \cite{GK}), and the proof of the Claim is concluded.

We now complete the proof that $S$ is invariant for optimal trajectories. To this aim, fix $\bar{x} \in S$, together 
with $\bar{\zeta} \in \mathbb{S}^1\cap N_{\mathcal{R}_{T(\bar{x})}}$ such that $h(\bar{x},\bar{\zeta}) = 0$. By the above 
claim, the never vanishing adjoint vector $\tilde{\lambda}(\cdot)$ which is the solution of (\ref{adjtilde}) is such that 
$h(\tilde{x}(t),\tilde{\lambda}(t)) = 0$ and $\tilde{\lambda}(t) \in N_{\mathcal{R}_t}(\tilde{x}(t))$ for all $t \in [0,T(\bar{x})]$, 
which shows that each point $\tilde{x}(t)$ of the optimal trajectory $\tilde{x}(\cdot)$ steering the origin to 
$\bar{x}$ belongs to $S$. The prove of the invariance of $S$ is complete.
\end{proof}
\begin{Theorem}\label{rect2dim}
Under the same assumptions of Proposition \ref{Sinv}, the set $S$ is countably $\mathcal{H}^1$-rectifiable.
Moreover, for all $\bar{x}\in S$ there exists $\delta>0$ such that 
\begin{equation}\label{propagnonlin}
\mathcal{H}^1\big(S\cap B(\bar{x},\delta)\big) >0.
\end{equation}
\end{Theorem}
\begin{proof}
In order to prove the rectifiability property of $S$, it is enough to show that, if $S$ is nonempty, then it consists exactly 
of two optimal trajectories of the reversed dynamics
\begin{equation}\label{revdyn2}
\left\{\begin{array}{ll}
\dot{x}(t)\: = \: -F(x(t))-G(x(t))u(t), \,u\in [-1,1]^M, t\in [0,\tau],\\
x(0)  \: = \:  0.
\end{array}\right.
\end{equation}
Let $\bar{x} \in S$ together with $\bar{\zeta} \in \mathbb{S}^1\cap N_{\mathcal{R}_{T(\bar{x})}}$ be such that $h(\bar{x},\bar{\zeta}) = 0$.
Let $\tilde{u}(\cdot)$ be the optimal control steering the origin to $\bar{x}$ and let $\tilde{x}(\cdot)$ (resp., $\tilde{\lambda}(\cdot)$) 
be the corresponding optimal trajectory (resp., adjoint vector, the solution of (\ref{adjtilde})). Set $\zeta_0 = \tilde{\lambda}(0) \ne 0$.

We assume now that $M = 1$, i.e., the control is scalar. 
Since the Hamiltonian is constant along the optimal trajectory $\tilde{x}$, we have that
\[
|\langle G(0),\zeta_0\rangle| = h(0,\zeta_0) = 0\; (= h(0,-\zeta_0)).
\]
We now prove that each one of the vectors $\zeta_0$ and $-\zeta_0$ determines uniquely an optimal trajectory of (\ref{revdyn2}) contained in $S$. 
In fact, for every optimal trajectory $x(\cdot)$ of (\ref{revdyn2}), with a corresponding adjoint vector $\lambda(\cdot)$, we can define 
the switching function
\[
g^+_{x,\lambda}(t) = \langle -G(x(t)),\lambda(t)\rangle.
\]
Of course, $g^+_{x,\lambda}(0) = \langle -G(0),\pm\zeta_0\rangle = 0$ and $\dot{g}^+_{x,\lambda}(0) = \mp \langle DF(0)G(0),\zeta_0 \rangle$. 
The last expression is nonzero, due to the assumption 3), so that in a neighborhood of $t=0$, the sign of $g^+_{x,\lambda}(\cdot)$ is uniquely 
determined by $\pm \zeta_0$. Therefore, in a neighborhood of $t=0$ the optimal control is uniquely determined by 
$\mathrm{sign}(g^+_{x,\lambda}(\cdot))$, by the Maximum Principle, and so there are exactly two optimal trajectories of (\ref{revdyn2}) 
which belong to $S$ in a neighborhood of $t=0$. Since at every zero of $g^+_{x,\lambda}(\cdot)$ the derivative $\dot{g}^+_{x,\lambda}(\cdot)$ 
is nonvanishing (see, \cite[Sections 3.2 and 5]{GK}), the optimal control can be uniquely extended up to the time $t=\tau$. 
The proof is now complete for the case of a single input.

To conclude the proof of the rectifiability, let $M = 2$. The condition $h(0,\zeta_0)=0$ means that the system of equations
\begin{equation*}
\left\{\begin{array}{ll}
\langle G_1(0),\zeta_0 \rangle\: = 0,\\
\langle G_2(0),\zeta_0 \rangle\: = 0.
\end{array}\right.
\end{equation*}
has nontrivial solutions. So, if $G_1(0)$ and $G_2(0)$ are linearly independent, then $S$ is empty. Otherwise, both components 
of the optimal controls are uniquely determined by the sign of the corresponding switching functions, exactly as for the single input case.

The propagation property \eqref{propagnonlin} is an immediate consequence of Proposition \ref{Sinv}.
The proof is concluded.
\end{proof}

\section{The SBV regularity of $T$}\label{sec:SBV}
As a consequence of the results contained in Sections 4 and 5 we prove the SBV regularity of the minimum time function $T$. 
We recall first some properties of functions with bounded variation, and next we collect some known results on functions having 
epigraph with positive reach. As it was proved in \cite{CMW} and in \cite{GK}, the minimum time function has this property under 
the assumptions taken in Section 4 or in Section 5.
 
Let $\Omega \subset \R^N$ be open. We say that a function $f \in L_{\mathrm{loc}}^1(\Omega)$ has locally bounded variation, and 
we denote this fact by $f \in BV_{\mathrm{loc}}(\Omega)$, if for every ball $\Delta \subset \Omega$ the distributional derivative 
of $f$ in $\Delta$ is a finite Radon measure (see, e.g., \cite[Definition 3.1]{AFP}), which we denote by $Df$. 
We write $Df = D^af + D^sf$, where $D^af$ is absolutely continuous with respect to Lebesgue measure, and $D^sf$ is singular. 
The singular part $D^sf$ can also be decomposed into the jump part, $D^jf$, and the Cantor part, $D^cf$ (see, \cite[Section 3.9]{AFP}). 
In the case where $f$ is continuous, like in the case $f = T$ under our assumptions, the jump part obviously vanishes.
 
\begin{Definition} \label{defSBV} (see, e.g., \cite[Section 4.1]{AFP}) We say that $f\in BV_{\mathrm{loc}}(\Omega)$ is a 
special function of locally bounded variation, $f\in SBV_{\mathrm{loc}}(\Omega)$, if the Cantor part of its derivative $D^cf$ vanishes.
\end{Definition}
It is our aim, in this section, to prove that under the assumptions of Section 4 and 5, the Cantor part $D^cT$ vanishes, 
and so $T$ is a special function of locally bounded variation.

We state some further results.

\begin{Proposition} \label{propSBV} (see \cite[Proposition 4.2]{AFP}) Let $f \in BV(\Omega)$. Then $f \in SBV(\Omega)$ if 
and only if $D^sf$ is concentrated on a Borel set $\sigma$-finite with respect to $\mathcal{H}^{N-1}$, in particular, if 
it vanishes outside a countably $\mathcal{H}^{N-1}$-rectifiable set.
\end{Proposition}
\noindent Recalling the definition of non-Lipschitz points given in Section 3 (see Definition \ref{definonlip}), we obtain the following result.
The notation $\mu_{\lfloor E}$ means the restriction of the measure $\mu$ to the set $E$.

\begin{Proposition}\label{SBVnonlip} Let $\Omega \subset \R^N$ be open and let $f\in SBV_{\mathrm{loc}}(\Omega)$. 
Let $$K = \{x \in \Omega: f\text{ is non-Lipschitz at } x\}.$$ Then $D^sf_{\lfloor\Omega \setminus K} = 0$.
\end{Proposition}
\begin{proof}
By definition, $f$ is locally Lipschitz in the open set $\Omega \setminus K$. Therefore $Df$ is absolutely continuous
with respect to Lebesgue measure $\mathcal{L}^N$ in $\Omega \setminus K$ (see, e.g., \cite[Proposition 2.13]{AFP}), i.e.,
$D^sf_{\lfloor\Omega \setminus K}=0$.
\end{proof}
\noindent Consequently, by putting together the two previous Propositions, we obtain

\begin{Corollary}\label{coSBV}  Let $\Omega \in \R^N$ be open and let $f\in BV_{\mathrm{loc}}(\Omega)$. 
Assume that the set of non-Lipschitz points of $f$  be countably $\mathcal{H}^{N-1}$-rectifiable. Then $f\in SBV_{\mathrm{loc}}(\Omega)$.
\end{Corollary}

\noindent We are now ready for the main results of this section.

\begin{Theorem}\label{TSBVlin} Consider the linear control system (\ref{linsys}) under the assumption (\ref{kalman}). 
Then the minimum time function $T$ to reach the origin satisfies $T \in SBV_{\mathrm{loc}}(\R^N)$.
\end{Theorem}
\begin{Corollary}\label{sobolevlin}
Let $1\le J\le N$ be the smallest integer such that 
\[
\mathrm{rk}\Big[ B, AB, \ldots, A^{J-1}\Big] = N.
\]
Then for every $1\le p < N(N-J+1)/(N-J)$ we have
\[
T\in W^{1,p}_{\mathrm loc}\big( \mathbb{R}^N\big). 
\]
\end{Corollary}

\begin{Theorem} \label{TSBVnonlin}
Consider the nonlinear system (\ref{nonlinsys}) under the assumptions 1) -- 4) stated in Section 5. Then there exists 
$\mathcal{T}>0$ depending only on $G(0)$, $DF(0)$, and on the Lipschitz constant $L$ of $DF$ and $DG$, such that 
$T \in SBV_{\mathrm{loc}}(\mathrm{int}(\mathcal{R}_{\mathcal{T}}))$.
\end{Theorem}

\noindent \textit{Proof of Theorem \ref{TSBVlin} and \ref{TSBVnonlin}.}
The statements follow immediately by putting together Theorem \ref{fposreach} and Theorem \ref{N-1rect} 
(resp., Theorem \ref{rect2dim}) and Corollary \ref{coSBV}.
\qed

\medskip

\noindent \textit{Proof of Corollary \ref{sobolevlin}.}
Since $T$ is continuous and belongs to $SBV_{\mathrm{loc}}(\R^N)$ then its distributional derivative $DT$ is a locally summable function.
Moreover, it is well known that $T$ is H\"older continuous with exponent $1/J$ (see, e.g., \cite[Theorem IV.1.9]{BCD} and references therein).
The statement then follows by applying standard results on Sobolev spaces
(see, e.g., Theorem 3, p. 277, in \cite{eva}).
\qed

\section{Propagation of non-Lipschitz singularities}\label{sec:propag}
This section deals with a lower estimate of the dimension of $\mathcal{S}$ for the linear case.
We show that the $\mathcal{H}^{N-k}$-rectifiability of $\mathcal{S}$ is indeed
optimal, in the sense of Theorems \ref{propagN-2} and \ref{propagN-1} below, at least for a small time. Those statements can be seen as propagation
results for singularities of non-Lipschitz type for the minimum time function.\\
We consider the linear system \eqref{linsys} under the assumptions \eqref{kalman} and \eqref{rkB}.
Let $N > 2$ and define, for $\tau >0$, $\mathcal{S}(\tau) = \mathcal{S} \cap \mathrm{bdry}\mathcal{R}_\tau$. We assume that
$k\le N-1$, otherwise $\mathcal{S}=\emptyset$.
\begin{Theorem}\label{propagN-2} There exists $\tilde{\tau} >0$, depending only on $A, B, N$, satisfying the following the property:
for all $\tau\le\tilde{\tau}$ and all $\mathcal{H}^{N-2}$-a.e. $x\in \mathcal{S}(\tau)$ there exists a neighborhood $V$ of $x$ such that
\begin{equation}\label{propag}
\mathcal{H}^{N-1-k} (V\cap \mathcal{S}(\tau)) > 0.
\end{equation}
\end{Theorem}
\begin{proof}
We divide the proof into some steps.\\
We will use a result which was proved, e.g., in \cite{Bru}
(see the proof of Lemma 8). The statement is as follows.
\begin{equation}\label{fact}
\begin{split}
\text{There exists $\bar{\tau} >0$, depending only on $A, B, N$, such that for every $\zeta \in \mathbb{S}^{N-1}$ the switching}\\
\text{function $g(\cdot,\zeta) = \langle \zeta, e^{-A\cdot}b \rangle$ has at most $N-1$ zeros in $[s,s+\bar{\tau}]$ for every $s\ge 0$.}\qquad\qquad\;\;\,
\end{split}
\end{equation}

\smallskip

\noindent
\textit{Claim 1}. The statement of the Theorem holds true in the case $B = b$, a vector.\\
\textit{Proof of Claim 1}. Let $\bar{\tau}$ be given by \eqref{fact}. Fix $0<\tau < \bar{\tau}$ and let $0 <j \le N-2$.
We say that $x\in \mathcal{S}(\tau)$ belongs to $\mathcal{S}_j(\tau)$ if
there exist times $0 < s_1 < s_2 < \cdots < s_j < \tau$ such that
$$
x= \pm \int_{0}^{\tau} {e^{-As}b\gamma(s)ds},
$$
where
\begin{equation*}
\gamma(s) = \left\{ \begin{array}{ll}
\,\,1 & \textrm{if} \,\, 0 < s < s_1,\\
-1 & \textrm{if} \,\, s_1 < s < s_2,\\
\cdots \\
 (-1)^j & \textrm{if} \,\, s_j < s < \tau.
\end{array} \right.
\end{equation*}
In other words, the optimal control steering the origin to $x$ for the reversed dynamics has exactly $j$ switchings in the interval $(0,\tau)$.

\smallskip

\noindent\textit{Step 1}. There exists $\tilde{\tau} > 0$, depending only on $A, b$ and $N$, such that if
$0 \le s_1 < s_2 < \cdots < s_j \le \tilde{\tau}$ then
\begin{equation}\label{indep}
\mathrm{rank}\left[ e^{-As_1}b, e^{-As_2}b, \cdots, e^{-As_j}b  \right] = j.
\end{equation}
In order to prove \eqref{indep}, for $s \ge 0$ and $\zeta\in\mathbb{R}^N$ set $g(s)= \langle\zeta , e^{-As}b\rangle$
and $H=\{ \zeta\in \mathbb{R}^N: g(s_i,\zeta)=0,i=1,\ldots ,j\}$, $0 \le s_1 < s_2 < \cdots < s_j \le \tau$. 
We claim that $\mathrm{dim}\, H=N-j$. Indeed, if $g(s_1,\zeta) =0$, then
\[
0= \langle \zeta , b - As_1 b\rangle + o(\tau),
\]
so that $\langle\zeta ,b\rangle =0$ since $\tau$ can be chosen small enough. Furthermore, if $j>1$, there exists $\bar{s}_1\in (s_1,s_2)$ such that
$\frac{\partial}{\partial s}g(\bar{s}_1,\zeta) =0$, which in turn implies
\[
0= - \langle \zeta , Ab - A^2\bar{s}_1 b\rangle + o(\tau),
\]
so that $\langle\zeta ,Ab\rangle =0$ since $\tau$ can be chosen small enough. The same argument provides times $\bar{s}_i$, $i=2,\ldots j-1$ such that
\[
0= \frac{\partial^i g}{\partial s^i}(\bar{s}_i,\zeta)= \langle \zeta , A^i b\rangle + O(\tau),
\]
i.e., $\langle \zeta , A^i b\rangle=0$. The proof is completed by invoking the rank condition \eqref{kalman}.

\smallskip

\noindent
\textit{Step 2}. If $ 0 < j \le N-2$, then for all $0<\tau <\tilde{\tau}$ the set $\mathcal{S}_j(\tau)$ is the union of two smooth parametrized $j$-surfaces.
Actually we are going to prove that  
$\{x\in \mathcal{S}_j(\tau) : x=  \int_{0}^{\tau} {e^{-As}b\gamma(s)ds}\} $ is a smooth parametrized $j$-surface,
the other case being entirely analogous.\\
Indeed, we have
\[
x=  \int_{0}^{s_1} {e^{-As}bds} + \sum_{i=1}^{j-1}(-1)^i  \int_{s_i}^{s_{i+1}} {e^{-As}bds}+ (-1)^j  \int_{s_j}^{\tau} {e^{-As}bds},
\]
where $0 < s_1 < \ldots < s_j <\tau$.
Observe that $ \frac{\partial x}{\partial s_i} = 2(-1)^{i+1}e^{-As_i}b, 1\le i \le j$, and by (\ref{indep}) the matrix
$\left(\frac{\partial x}{\partial s_i} \right)_{i=1,\cdots,j}$ has rank $j$ in the open set $\{ (s_1,\ldots,s_j)\in (0,\tau)^j : s_1 < \ldots < s_j \}$.
The proof of Step 2 is concluded.
 
\smallskip
 
\noindent Set now $\mathcal{S}_0(\tau) = \left\{ \pm \int_{0} ^{\tau} {e^{-As}bds}\right\}$.
By the Maximum Principle, owing to \eqref{fact} we have that
\[
\mathcal{S}(\tau) =\bigcup_{j=0}^{N-2}  \mathcal{S}_j(\tau)
\]
for all $0<\tau<\tilde{\tau}$ and the union is disjoint.
In particular, Step 2 implies that for all such $\tau$
\begin{equation}\label{fullmeas}
\mathcal{H}^{N-2}\big( \mathcal{S}(\tau) \setminus \mathcal{S}_{N-2}(\tau)\big) = 0
\end{equation}
and that (\ref{propag}) holds at every point $x \in \mathcal{S}_{N-2}(\tau)$.
The proof of Claim 1 is concluded.

\smallskip

\noindent
\textit{Claim 2}. The statement of Theorem \ref{propagN-2} holds in the general case.\\
\textit{Proof of Claim 2}. Let $0<\tau<\tilde{\tau}$ be given and fix $x\in \mathcal{S}\cap\mathrm{bdry}\, \mathcal{R}_\tau$,
together with the optimal control $u=(u_1,\ldots ,u_M)$ steering
the origin to $x$ in time $\tau$ by the reversed dynamics. Assume that $u_i$ has exactly $\kappa_i+1$ zeros, $0\le \kappa_i \le N-2$, at times
\[
0= s_0^i < s_1^i < \ldots < s_{\kappa_i}^i \le\tau.
\]
Then, recalling \eqref{rkB}, we have
\[
k\le  \mathrm{rank}\big\{ b_i,e^{As_1^i}b_i,\ldots ,e^{As_{\kappa_i}^i}b_i : 1\le i\le M \big\}\le N-1.
\]
For $j=0,\ldots , N-(1+k)$, let $\mathcal{S}_{j+k}(\tau)$ be the set of all $x\in\mathcal{S}(\tau)$ such that
\[
 \mathrm{rank} \big\{ e^{As_1^i}b_i,\ldots ,e^{As_{\kappa_i}^i}b_i : 1\le i\le M \big\}=j+1.
\]
Observe that
\[
\mathcal{S}(\tau)= \bigcup_{j=0}^{N-(1+k)} \mathcal{S}_{j+k}(\tau)
\]
and the union is disjoint. Moreover, by arguing exactly as in Steps 1 and 2 in the proof of Claim 1 above, we can see that each $\mathcal{S}_{j+k}(\tau)$
is a union of finitely many disjoint smooth parametrized $j$-surfaces. Thus,
\[
\mathcal{H}^{N-k-1}\big( \mathcal{S}_{j+k}(\tau) \big) =0\qquad \forall j=0,\ldots , N-(2+k)
\]
and
\[
\mathcal{H}^{N-k-1}\big( \mathcal{S}_{N-1}(\tau) \big) >0.
\]
Therefore, for $\mathcal{H}^{N-(k+1)}$-a.e. $x\in\mathcal{S}(\tau)$ there exists a neighborhood $V=V(x)$ such that
\[
\mathcal{H}^{N-(k+1)}\big( V\cap\mathcal{S}(\tau)\big)>0. 
\]
The proof is now complete.
\end{proof}
By combining the above result with the invariance statement contained in Theorem \ref{charnonlip} we obtain immediately the following
\begin{Theorem}\label{propagN-1} Let $\bar{\tau}$ be given as in Theorem \ref{propagN-2}.
Then for $\mathcal{H}^{N-k}$-a.e. $x\in \mathcal{S}$ such that $T(x) < \bar{\tau}$ there exists a neighborhood $V$ of $x$ such that
\[
\mathcal{H}^{N-k} (V\cap \mathcal{S}) > 0.
\]
\end{Theorem}
\begin{proof}
Fix $0<\tau < \bar{\tau}$ and let $E$ be a subset of $\mathcal{S}\cap \mathrm{bdry}\mathcal{R}_{\tau}$ with full $\mathcal{H}^{N-(k+1)}$-measure
with the property (\ref{propag}). Then the optimal trajectories for the reversed dynamics through each point of $E$ from a
subset of $\mathcal{S}$ with full $\mathcal{H}^{N-k}$-measure. The proof is complete.
\end{proof}

\noindent\textbf{Remark.} The statement of theorems \ref{propagN-2} and \ref{propagN-1} are somewhat unnatural for linear systems, as they
are valid only for small times. The proof for arbitrarily large times requires an analysis of higher order and of linearly dependent zeros
of the switching function, which we are not yet able to conclude.

\smallskip

\noindent Observe finally that in the nonlinear two dimensional case the propagation result is contained in (\ref{propagnonlin}).

\end{document}